\documentclass{elsarticle}

\usepackage{amsrefs}
\usepackage{amsmath, amssymb}
\usepackage{amsthm}
\usepackage[normalem]{ulem}
\usepackage[all,cmtip]{xy}
\usepackage{graphics}
\usepackage{caption}
\usepackage{subfig}
\usepackage{algorithm}
\usepackage{algorithmic}
\usepackage{booktabs}
\usepackage[usenames, dvipsnames]{color}
\usepackage{cancel}

\def\rmd{{\rm d}}
\def\rmT{{\text{T}}}

\newtheorem{theorem}{{\bf Theorem}}
\newtheorem{corollary}[theorem]{{\bf Corollary}}
\newtheorem{proposition}[theorem]{{\bf Proposition}}
\newtheorem{lemma}[theorem]{{\bf Lemma}}
\newtheorem{conjecture}[theorem]{{\bf Conjecture}}
\theoremstyle{definition}

\newtheorem{example}{{\bf Example}}

\def\bfb{{\boldsymbol{b}}}
\def\bfx{{\boldsymbol{x}}}

\def\bfc{{\boldsymbol{c}}}
\def\bft{{\boldsymbol{t}}}
\def\bfu{{\boldsymbol{u}}}
\def\bfv{{\boldsymbol{v}}}
\def\tbfx{{\tilde{\bfx}}}
\def\ta{{\tilde{a}}}
\def\tb{{\tilde{b}}}
\def\tx{{\bf x}}
\def\ty{{\bf y}}

\def\mA{{\boldsymbol{A}}}
\def\mB{{\boldsymbol{B}}}
\def\mC{{\boldsymbol{C}}}

\def\mM{{\boldsymbol{M}}}
\def\mQ{{\boldsymbol{Q}}}

\def\mI{{\boldsymbol{I}}}

\def\CC{\mathbb{C}}
\def\NN{\mathbb{N}}
\def\QQ{\mathbb{Q}}
\def\RR{\mathbb{R}}

\def\CCC{\mathcal{C}}
\def\PPP{\mathcal{P}}

\def\SSS{\mathcal{S}}

\def\rmd{\mathrm{d}}
\def\Res{\mathrm{Res}}
\def\sgn{\mathrm{sgn}}

\makeatletter
\def\ps@pprintTitle{%
  \let\@oddhead\@empty
  \let\@evenhead\@empty
  \def\@oddfoot{\reset@font\hfil\thepage\hfil}
  \let\@evenfoot\@oddfoot
}
\makeatother

\begin{document}

\begin{frontmatter}
\title{Similarity detection of rational space curves}

\author[a]{Juan Gerardo Alc\'azar\fnref{proy}}
\ead{juange.alcazar@uah.es}
\author[a]{Carlos Hermoso}
\ead{carlos.hermoso@uah.es}
\author[b]{Georg Muntingh}
\ead{georg.muntingh@sintef.no}

\address[a]{Departamento de F\'{\i}sica y Matem\'aticas, Universidad de Alcal\'a,
E-28871 Madrid, Spain}
\address[b]{SINTEF ICT, PO Box 124 Blindern, 0314 Oslo, Norway}

\fntext[proy]{Supported by the Spanish ``Ministerio de
Ciencia e Innovacion" under the Project MTM2014-54141-P. 
Member of the Research Group {\sc asynacs} (Ref. {\sc ccee2011/r34}) }

\begin{abstract}
We provide an algorithm to check whether two rational space curves are related by a similarity. The algorithm exploits the relationship between the curvatures and torsions of two similar curves, which is formulated in a computer algebra setting. Helical curves, where curvature and torsion are proportional, need to be distinguished as a special case. The algorithm is easy to implement, as it involves only standard computer algebra techniques, such as greatest common divisors and resultants, and Gr\"obner basis for the special case of helical curves. Details on the implementation and experimentation carried out using the computer algebra system Maple 18 are provided. 
\end{abstract}

\end{frontmatter}

\section{Introduction}\label{section-introduction}
\noindent Two objects are \emph{similar} when one of them is the result of applying an isometry and scaling to the other. Therefore, two similar objects have the same shape, although their position and size can be different. Because of this, recognizing similar objects is important in the field of Pattern Recognition, where one typically has a database of objects and wants to compare, up to a similarity, a given object with all the elements in the database. 

Three-dimensional similarity detection is also important in Computer Graphics and Computer Vision, and therefore it has been addressed in a long list of papers. Following the introduction of \cite{Chen}, the methods proposed in these papers can be grouped into two different categories: \emph{shape-based} and \emph{topology-based}. In the first category, one picks \emph{feature descriptors} for the objects to be checked, giving rise to feature vectors that are later compared using appropriate metrics; see for instance the survey \cite{B06} or the papers \cite{ank99, koertgen, osada}. In the second category, which has gained attention in recent years, a ``skeleton" is computed from each object, which is later used for comparison purposes; see \cite{H01, Sundar}. The aforementioned papers, and others that can be found in their bibliographies, focus on surfaces, upon which (almost) no structure is assumed. At most, some of these papers require the objects to be modeled by means of polyhedra, so that they are considered to be meshings of perhaps more complex shapes. Additionally, in these references similarity detection is usually considered only up to a certain tolerance, so that the criteria are approximate. 
	
Our approach is different. First, we deal with exact one-dimensional objects with a strong structure, namely rational space curves defined by rational parametrizations. Furthermore, we exploit the structure of the space curves to check, in a deterministic fashion, whether they are similar, and to explicitly compute the similarities between both curves in the affirmative case. In order to do this, we build on previous work on similarities of plane curves \cite{AHM14} and symmetries of plane and space curves \cite{AHM14-2, AHM15}. As in these papers, we exploit the rationality of the curves to reduce the problem to the parameter space. Analogously to the algorithm in \cite{AHM15}, the algorithm in this paper is based on comparing curvatures and torsions. However, similarity has the additional substantial difficulty of determining the scaling. Interestingly, this forces us to distinguish as a special case the \emph{helical curves}, i.e., space curves with proportional curvature and torsion.  

The basic steps in the algorithm are as follows. If the two given rational space curves are similar, then there exists a rational function relating the parameter spaces conforming to the similarity between the ambient spaces of the curves. Under the hypothesis that the parametrizations of the curves are proper, i.e., injective for almost all points, this rational function is a M\"obius transformation. In our algorithm one first computes candidates for the scaling constants and then candidates for the M\"obius transformations. After this, the similarities between the curves can be computed. If the input curves are non-helical, then we have two independent conditions involving the curvatures and torsions of the curves, and from these conditions the scaling constant can be found. If the input curves are helical, then these two conditions are no longer independent, and a different approach based on a procedure in \cite{AHM15} is provided.

As for plane curves \cite[\S 3.5]{AHM14}, the method can be adapted to the case of piecewise rational space curves. Moreover, for a space of properly parametrized curves satisfying affine invariance and uniqueness of the control polygon, we show that detecting similarity of such curve segments reduces to detecting similarity of the control polygons. This includes B\'ezier curves and, under certain conditions, B-spline curves and NURBS curves.

The structure of the paper is as follows. In Section \ref{sec-gen} we provide some background on isometries, similarities, differential invariants and helical curves, and we prove some results that are needed later in the paper. Section \ref{sec:detecting} describes the algorithm for solving the problem, separately considering the case of non-helical and helical curves. In Section \ref{sec:experimentation} we report on the experimentation with the algorithm, implemented in the computer algebra system Maple 18. In Section~\ref{segments} we briefly discuss similarity detection of curve segments. Finally, conclusions and future work are presented in Section \ref{sec:conclusion}. 

\section*{Acknowledgments}
\noindent 
This paper provides a more detailed presentation of the theoretical aspects and a new section on experimentation, compared to the 8-page abridged form published in the proceedings \cite{AHM16} of the ISSAC 2016 meeting, where these results were presented. We are grateful to the audience and referees for their valuable feedback.

\section{Background} \label{sec-gen}

\subsection{Similarities and isometries of Euclidean space}
\noindent A \emph{similarity} of Euclidean space is a linear affine map from the space to itself that preserves ratios of distances. Equivalently, a map $f: \RR^3 \longrightarrow \RR^3$ is a similarity if and only if
\begin{equation}\label{eq:similarity}
f(\tx) = \lambda \mQ \tx + \bfb,\quad 0 \neq \lambda \in \RR,\ \bfb\in \RR^3,\ \mQ\in \RR^{3\times 3},\ \mQ^\rmT\mQ = \mI,\ \det(\mQ) = 1,
\end{equation}
where the latter two conditions mean that $\mQ$ is a special orthogonal matrix, i.e., a rotation about a line. Equivalently, with $\|\tx\|$ denoting the Euclidean norm of $\tx$ and $d(\tx,\ty) := \|\tx - \ty\|$ the Euclidean distance,
\begin{equation}\label{eq:similarity2}
d\big(f(\tx),f(\ty)\big) = |\lambda|\cdot d(\tx,\ty),\qquad \tx, \ty\in \CC^3.
\end{equation}
We refer to $\lambda$ as the \emph{(signed) ratio} of the similarity. A similarity is said to \emph{preserve} the orientation if $\lambda > 0$, and \emph{reverse} the orientation if $\lambda < 0$.  The identity map $f(\tx) = \tx$ is called the \emph{trivial similarity}.

If $|\lambda| = 1$ then $f$ is an \emph{(affine) isometry}, i.e., $f$ preserves distances. The classification of nontrivial isometries includes reflections (in a plane), rotations (about an axis), and translations, and these combine in commutative pairs to form twists, glide reflections, and rotatory reflections. More precisely, a \emph{twist} is the composition of a rotation about an axis and a translation in the direction of a vector parallel to this axis, while a \emph{glide reflection} is the composition of a reflection in a plane and a translation in the direction of a vector parallel to this plane. A composition of three reflections in mutually perpendicular planes through a point $\tx$ yields a \emph{central inversion} (with respect to the point $\tx$). The particular case of rotation by an angle $\pi$ is of special interest, and it is called a \emph{half-turn}.

If $\lambda$ is not an eigenvalue of $\mQ$, then $f$ has a unique fixed point $\bfc := (\mI - \lambda \mQ)^{-1}\bfb$, called the \emph{center} of the similarity. In particular any similarity that is not an isometry has a center, because $\mQ$, being orthogonal, has eigenvalues of modulus equal to 1. A \emph{dilatation} is a special type of similarity, defined as a map that sends any line to a parallel line (which could be the original line). Any dilatation that is not a translation sends any point $\tx$ to $\bfc + \lambda(\tx - \bfc)$ and therefore takes the form $f(\tx) = \lambda \mI \tx + (1-\lambda)\bfc$. A \emph{dilative rotation} is a composition of a dilatation $f$ with center $\bfc$ with a rotation $\mQ$ about a line $\ell$ containing $\bfc$, which takes the form
\begin{equation}\label{eq:dilativerotation}
f(\mQ \tx) = \mQ f(\tx) = \lambda \mQ \tx + (1-\lambda)\bfc,\qquad \mQ\bfc = \bfc.
\end{equation}
We recall the following characterization of similarities from \cite[p. 103]{Coxeter}.

\begin{theorem}\label{thm:similarityclassification}
Any similarity is either an isometry or a dilative rotation.
\end{theorem}

Similarities form a group under composition, and isometries form a subgroup of this group.

\subsection{Similarities and symmetries of rational space curves} \label{sec-prelim}
\noindent Theoretical aspects of the equivalence problem for space curves can be traced back to \'Elie Cartan \cite{Cartan}. Here we consider this problem for two rational space curves $\CCC_1, \CCC_2 \subset \RR^3$, neither lines nor circles, assumed to be nonplanar unless specified otherwise. Such curves are irreducible and can be parametrized by rational maps
\begin{equation}\label{eq:parametrizations}
{\bfx}_j: \RR \dashrightarrow \CCC_j\subset \RR^3, \qquad \bfx_j(t) = \big( x_j(t), y_j(t), z_j(t) \big),\qquad j = 1,2.
\end{equation}
As the components $x_j, y_j, z_j$ of $\bfx_j$ are rational functions of $t$ with real coefficients, they are defined for all but a finite number of values of $t$. We assume that the parametrizations \eqref{eq:parametrizations} are \emph{proper}, i.e., birational or, equivalently, injective except for perhaps finitely many values of $t$. This can be assumed without loss of generality, since any rational curve can quickly be properly reparametrized. For these claims and other results on properness, the interested reader can consult \cite{SWPD} for plane curves and \cite[\S 3.1]{A12} for space curves. We also assume that the numerators and denominators of the components of $\bfx_j$ are relatively prime. 

This paper concerns algebraic space curves \eqref{eq:parametrizations} that are \emph{similar}, i.e., one is the image of the other under a similarity. We say that $\CCC_1$ and $\CCC_2$ are \emph{related by a similarity $f$} when $f(\CCC_1) = \CCC_2$. We first establish some basic properties. 

\begin{lemma}[See {\cite[Lemma 1]{AHM15}}]\label{lem:invariantisometries}
A rational space curve different from a line cannot be invariant under a translation, glide reflection, or twist.
\end{lemma}

Therefore, reflections, rotations, and their combinations are the only isometries that leave a rational space curve different from a line invariant.

\begin{lemma} \label{center}
Let $f$ be a nontrivial similarity that is not an isometry, leaving an algebraic space curve $\CCC$ invariant. Then its center $\bfc$ is a point of $\CCC$. 
\end{lemma}

\begin{proof}
Since $f$ is not an isometry, $|\lambda| \neq 1$. If $|\lambda| > 1$ then $f^{-1}$ is a similarity with ratio $\lambda^{-1}$ satisfying $|\lambda ^{-1}| < 1$, also leaving $\CCC$ invariant. Therefore we can and will assume $|\lambda| < 1$. Let $\tx\in \CCC$. Since $f(\CCC)=\CCC$, the entire orbit $\{ \tx, f(\tx), f^2(\tx), \ldots \} \subset \CCC$. Using \eqref{eq:similarity2}, $|\lambda| < 1$, and $f(\bfc) = \bfc$, we have
\[ \lim_{k\to \infty} d\big(\bfc, f^k (\tx)\big) = \lim_{k\to \infty} |\lambda|^k d(\bfc, \tx) = 0,\]
and $f^k(\tx)$ approaches $\bfc$. Since $\CCC \subset \RR^3$ is closed, this limit must be a point of~$\CCC$.
\end{proof} 

In addition, note that $\bfc$ is not an isolated point, since it is the limit of a sequence of points of $\CCC$.

\begin{lemma}\label{lem-comp}
Let $f$ be a similarity that is not an isometry. Then for any positive integer $n$, the $n$-fold composition $f^n$ is not an isometry either.
\end{lemma}
\begin{proof} By \eqref{eq:similarity}, $f^n$ is a similarity of ratio $\lambda^n$, and $|\lambda|\neq 1$ implies $|\lambda^n|\neq 1$.
\end{proof}

\begin{lemma}\label{lem-inv}
Let $f$ be a similarity such that there exist distinct vectors $\tx,\ty$ with $d\big(f(\tx),f(\ty)\big)=d(\tx,\ty)$. Then $f$ is an isometry.
\end{lemma}

\begin{proof} Since $d(\tx, \ty) = d\big(f(\tx),f(\ty)\big)\neq 0$, Equation \eqref{eq:similarity2} implies $|\lambda| = 1$.
\end{proof}

Analogous to \cite[Proposition 2]{AHM14} for algebraic plane curves, the following theorem states that a self-similarity of an algebraic space curve is an isometry.

\begin{theorem} \label{preserve}
Let $f$ be a similarity that leaves an algebraic space curve $\CCC$, which is not a union of (possibly complex) concurrent lines, invariant. Then $f$ is an isometry.
\end{theorem}

\begin{proof}
Suppose $f$ is not an isometry. By Theorem \ref{thm:similarityclassification}, the similarity is a dilative rotation $f(\tx) = \lambda \mQ \tx + (1-\lambda)\bfc$, with $|\lambda|\neq 1$ and $\mQ$ a rotation about a line $\ell$ containing $\bfc$. Let $\Pi$ be the plane through $\bfc$ normal to $\ell$. 

Since $f$ maps lines through $\bfc$ to each other, we can assume without loss of generality that $\CCC$ has no such components. First consider the case that $\CCC$ has one or more planar irreducible components that are not a real or complex line. Since a similarity maps planes to planes, one of these components $\CCC'\subset \CCC$ satisfies $f^n(\CCC') = \CCC'$ for some integer $n \geq 1$. Since $\CCC'$ is not a line, it spans a plane $\Pi'$ and $f^n$ restricts to a plane similarity $f' := f^n|_{\Pi'}$ with $f'(\CCC') = \CCC'$, which is an isometry by \cite[Proposition 2]{AHM14}. Hence $f^n$ is an isometry by Lemma \ref{lem-inv} and $f$ is an isometry by Lemma \ref{lem-comp}.

It remains to show the case where $\CCC$ does not have any planar irreducible components besides lines. Supposing $f$ is not an isometry, $\CCC$ cannot contain a line $L$ parallel to $\ell$, because then it would also contain any parallel line $f^n(L),n\in \NN$, of which there are infinitely many since each has a different distance to $\ell$. Therefore the image $\CCC^\perp$ of the orthogonal projection $p: \CCC\longrightarrow \Pi$ is a plane curve. Since $\CCC$ does not have any planar components, $\CCC^\perp$ does not have any lines. Moreover,
\[ f (\CCC^\perp) = f\circ p (\CCC) = p\circ f(\CCC) = p\circ \CCC = \CCC^\perp, \]
showing that the restriction $f|_\Pi$ is a plane similarity that leaves $\CCC^\perp$ invariant. It follows that $f|_\Pi$ is an isometry by \cite[Proposition 2]{AHM14} and that $f$ is an isometry by Lemma \ref{lem-inv}.
\end{proof}

A nontrivial isometry $f$ leaving an algebraic space curve $\CCC$ invariant is called a \emph{symmetry} of $\CCC$. The curve $\CCC$ is called \emph{symmetric} if it has a symmetry. For a background on symmetries of rational space curves, see \cite{AHM14-2, AHM15}. Analogously, two curves $\CCC_1, \CCC_2$ are said to be \emph{similar} if there exists a similarity $f$ such that $f(\CCC_1) = \CCC_2$. 

Although we state and prove the following result in the irreducible setting, which is the case for the rational curves studied in this paper, an analogous statement holds for reducible curves.
\begin{corollary} \label{uniqeness}
Let $\CCC_1,\CCC_2$ be similar irreducible algebraic space curves, neither a line or a circle. There are finitely many similarities $f$ such that $f(\CCC_1) = \CCC_2$. Moreover, such a similarity $f$ is unique if and only if $\CCC_1,\CCC_2$ are not symmetric.
\end{corollary}

\begin{proof}
Assume there are distinct similarities $f_1, f_2$ with $f_1(\CCC_1) = \CCC_2 = f_2(\CCC_1)$. Then $f_1\circ f_2^{-1}$ is a nontrivial similarity transforming $\CCC_1$ into itself. By Theorem~\ref{preserve}, $f_1\circ f_2^{-1}$ is a nontrivial isometry, and therefore a symmetry of $\CCC_1$. Since the number of symmetries of a space curve different from a line or a circle is finite \cite{AHM14-2}, the first part follows. As for the second part, if $\CCC_1$ is not symmetric then $f_1\circ f_2^{-1}$ is the identity, and $f_1 = f_2$. Conversely, if $\CCC_1$ has a symmetry $f$, then $f_1\circ f$ is another similarity from $\CCC_1$ to $\CCC_2$.
\end{proof}

\begin{proposition} \label{lambda-unique}
Let $\CCC_1,\CCC_2$ be irreducible algebraic space curves, not a union of concurrent (or parallel) lines, for which there exist similarities $f_i(\tx) = \lambda_i \mQ_i \tx + \bfb_i$ such that $f_i(\CCC_1) = \CCC_2$, with $i=1,2$. Then $|\lambda_1| = |\lambda_2|$.
\end{proposition}

\begin{proof} One has $f_2^{-1}(\tx) = \lambda_2^{-1}\mQ_2^{-1}(\tx - \bfb_2)$. Then $(f_2^{-1}\circ f_1)(\tx) = \lambda\mQ\tx + \bfb$, with
\[\lambda := \frac{\lambda_1}{\lambda_2},\qquad \mQ := \mQ_2^\rmT \mQ_1, \qquad \bfb := \frac{1}{\lambda_2} \mQ_2^\rmT(\bfb_1 - \bfb_2),
\]
is a similarity since $0 \neq \lambda \in \RR$, $\det(\mQ) = \det(\mQ_2^\rmT) \det(\mQ_1) = 1$, and
\[ \mQ^\rmT\mQ = \mQ_1^\rmT \mQ_2 \mQ_2^\rmT \mQ_1 = \mQ_1^\rmT \mI \mQ_1 = \mI.\]
Since $f_2^{-1} \circ f_1$ leaves $\CCC_1$ invariant, Theorem \ref{preserve} implies that $f_2^{-1}\circ f_1$ is an isometry, implying $|\lambda_1/\lambda_2| = 1$ and therefore $|\lambda_1| = |\lambda_2|$.
\end{proof}

It is well known that the birational functions on the line are the \emph{M\"obius transformations} \cite{SWPD}, i.e., rational functions
\begin{equation}\label{eq:Moebius}
\varphi: \RR\dashrightarrow \RR,\qquad \varphi(t) = \frac{a t + b}{c t + d},\qquad \Delta := ad - bc \neq 0.
\end{equation}

The following result relates the similarity $f$ in space to a M\"obius transformation on the line. In \cite{AHM14} a proof was given for the case of plane curves, which generalizes \emph{mutatis mutandis} to the case of space curves.

\begin{theorem}\label{th-fund}
Let $\CCC_1, \CCC_2 \subset \RR^3$ be rational space curves with proper parametrizations $\bfx_1, \bfx_2: \RR\dashrightarrow \RR^3$. If $\CCC_1$, $\CCC_2$ are related by a similarity $f$, then there exists a unique M\"obius transformation $\varphi$ for which the diagram
\begin{equation}\label{eq:fundamentaldiagram}
\xymatrix{
\CCC_1 \ar[r]^{f} & \CCC_2 \\
\RR \ar@{-->}[u]^{\bfx_1} \ar@{-->}[r]_{\varphi} & \RR \ar@{-->}[u]_{\bfx_2}
}
\end{equation}
is commutative. 
\end{theorem}

Since the M\"obius transformation $\varphi$ maps the real line to itself, its coefficients can always be assumed to be real by dividing by a common complex number if necessary \cite[Lemma 3]{AHM15}. Notice that $s=\varphi(t)$ provides the $s$-value generating the image, in $\CCC_2$, under the similarity $f$, of the point generated by $t$ in $\CCC_1$.

\begin{corollary} \label{arclength}
Consider proper parametrizations $\bfx_j$, $j=1,2$, as in \eqref{eq:parametrizations}, a similarity $f$ as in \eqref{eq:similarity}, and a M\"obius transformation $\varphi$, related by \eqref{eq:fundamentaldiagram}. Then
\begin{equation} \label{eq:arc}
|\lambda| \cdot \left\| \bfx_1'(t)\right\| - \left\| (\bfx_2\circ \varphi)'(t)\right\| = 0,
\end{equation}
\end{corollary}

\begin{proof}
The commutative diagram \eqref{eq:fundamentaldiagram} has the corresponding equation
\[ \lambda \mQ \bfx_1(t)+\bfb=(\bfx_2 \circ \varphi)(t).\]
Differentiating and taking norms yields $\|\lambda \mQ \bfx_1'(t)\| = \|(\bfx_2 \circ \varphi)'(t)\|$,
which, using the orthogonality of $\mQ$, yields \eqref{eq:arc}.
\end{proof}

\subsection{Differential invariants}\label{sec:differentialinvariants}
\noindent The remainder of the section concerns the effect of a similarity and M\"obius transformation on the \emph{curvature} $\kappa$ and \emph{torsion} $\tau$ of a parametric curve $\bfx$, which are defined by 
\begin{equation} \label{kt}
\kappa = \kappa_\bfx := \frac{\| \bfx' \times \bfx''\|}{\| \bfx' \|^3},  \qquad \tau = \tau_\bfx := \frac{\langle \bfx'\times \bfx'', \bfx'''\rangle}{\| \bfx' \times \bfx''\|^2}
\end{equation}
Notice in particular that $\kappa \geq 0$, while $\tau$ can be positive, negative, or zero. Moreover, although $\tau$ and $\kappa^2$ are rational functions for any rational map $\bfx$, the curvature $\kappa$ is in general not rational.

\begin{lemma}\label{lem:kappatauf}
For a similarity $f(\tx) = \lambda \mQ \tx + \bfb$ and parametrization $\bfx$ as in \eqref{eq:parametrizations},
\[ |\lambda| \cdot \kappa_{f\circ \bfx} = \kappa_\bfx,\qquad  \lambda\cdot \tau_{f\circ \bfx} = \tau_\bfx. \]
\end{lemma}
\begin{proof}
A straightforward calculation yields, for any invertible matrix $\mM\in \RR^{3\times 3}$ and vectors $\bfu, \bfv\in \RR^3$, the identity 
\begin{equation}\label{eq:orthogonalcrossproduct}
(\mM\bfu) \times (\mM\bfv) = \det(\mM) (\mM^{-1})^{\rmT} (\bfu \times \bfv).
\end{equation}
Using $(f\circ \bfx)^{(n)} = \lambda \mQ \bfx^{(n)}$ for $n = 1,2,3$ and $\det(\mQ)=1$ with $\mQ$ orthogonal,
\[ |\lambda|\cdot \kappa_{f\circ \bfx}
 = \frac{\| (\mQ\bfx') \times (\mQ\bfx'')\|}{\| \mQ\bfx' \|^3}
 = \frac{\| \mQ(\bfx' \times \bfx'')\|}{\| \mQ\bfx' \|^3}
 = \frac{\| \bfx' \times \bfx''\|}{\| \bfx' \|^3}
 = \kappa_\bfx, \]
\[ \lambda \cdot \tau_{f\circ \bfx}
 = \frac{\langle (\mQ\bfx')\times (\mQ \bfx''), \mQ\bfx'''\rangle}{\| (\mQ\bfx') \times (\mQ\bfx'')\|^2}
 = \frac{\langle \mQ(\bfx' \times \bfx''), \mQ\bfx'''\rangle}{\|  \mQ(\bfx' \times \bfx'')\|^2} 
 = \tau_\bfx.
 \qedhere \]
\end{proof}

Next we recall a lemma from \cite{AHM15}, which describes the behavior of the curvature and torsion under reparametrization, for instance by a M\"obius transformation.
\begin{lemma}\label{lem:kappatauphi}
Let $\bfx$ be a rational parametrization \eqref{eq:parametrizations} and let $\phi\in C^3(U)$, with $U\subset \RR$ open. Then
\[ \kappa_{\bfx \circ \phi} = \kappa_\bfx \circ \phi,\qquad \tau_{\bfx \circ \phi} = \tau_\bfx \circ \phi, \]
whenever both sides are defined.
\end{lemma}

The following lemma relates the curvatures and torsions of similar curves.

\begin{lemma}\label{lem:curvaturetorsionsimilar}
Suppose $\bfx_1,\bfx_2$ define curves $\CCC_1, \CCC_2$ with $f(\CCC_1) = \CCC_2$ for a similarity~$f$ with ratio $\lambda$. Then there is a M\"obius transformation $\varphi$ such that
\begin{equation}\label{eq:curvaturetorsionsimilar}
   \kappa_{\bfx_2} \circ \varphi
 = \kappa_{\bfx_2 \circ \varphi}
 = \kappa_{f\circ \bfx_1}
 = \frac{1}{|\lambda|} \kappa_{\bfx_1},\qquad
\tau_{\bfx_2} \circ \varphi
 = \tau_{\bfx_2 \circ \varphi}
 = \tau_{f\circ \bfx_1}
 = \frac{1}{\lambda} \tau_{\bfx_1}. \qedhere
\end{equation}
\end{lemma}
\begin{proof}
By Theorem \ref{th-fund}, there exist a M\"obius transformation $\varphi$ such that $f\circ \bfx_1 = \bfx_2 \circ \varphi$. The statement follows from Lemmas \ref{lem:kappatauf} and \ref{lem:kappatauphi}.
\end{proof}

\subsection{Helical curves} \label{subsec-hel}
\noindent Consider parametrizations $\bfx_i$, $i=1,2$, as in \eqref{eq:parametrizations} defining nonplanar curves. Then the torsion $\tau_{\bfx_i}$ is not identically zero, and we can consider the ratio
\[ \mu_i := \frac{\kappa_{\bfx_i}}{\tau_{\bfx_i}},\qquad i = 1,2. \]
Whenever this ratio is constant we refer to it as the \emph{proportionality constant}. Such nonplanar curves are called \emph{helical curves} \cite{Farouki, Fatma}, generalizing the familiar circular helix in which case not only the quotient of the curvature and torsion, but also the curvature and torsion themselves are constant.

\begin{lemma}\label{lem:nonzeromu}
Any rational helical curve $\bfx$ has proportionality constant $\mu\neq 0$.
\end{lemma}
\begin{proof}
Suppose $\mu = 0$. If $\bfx'\equiv 0$ or $\bfx''\equiv 0$, then integrating would yield a point or a line, which are planar and therefore non-helical. Therefore, since $\kappa \equiv 0$, there exists a nonzero function $\nu$ such that $\bfx''=\nu \cdot \bfx'$. Writing $\bfx = (x, y, z)$, integrating $x''/x' = \nu$, $y''/y' = \nu$, $z''/z' = \nu$ and taking exponentials yields $\bfx'(t) = \bfx_0\cdot \exp\big(\int \nu(t)\rmd t \big)$ for some constant vector $\bfx_0$. Therefore $\bfx$ is a line, contradicting that $\bfx$ is helical. We conclude $\mu \neq 0$.
\end{proof}

\begin{proposition}\label{prop:helicalequalmu}
Suppose $\bfx_1,\bfx_2$ define helical curves $\CCC_1, \CCC_2$ with proportionality constants $\mu_1, \mu_2$ satisfying $f(\CCC_1) = \CCC_2$ for a similarity~$f$ with ratio $\lambda$. Then $\mu_2 = \sgn(\lambda)\cdot \mu_1$.
\end{proposition}
\begin{proof}
Taking the quotient in \eqref{eq:curvaturetorsionsimilar} yields
\[ \mu_2
 = \frac{\kappa_{\bfx_2}}{\tau_{\bfx_2}} \circ \varphi
 = \frac{\kappa_{\bfx_2} \circ \varphi}{\tau_{\bfx_2} \circ \varphi}
 = \frac{\frac{1}{|\lambda|}}{\frac{1}{\lambda}}\cdot \frac{ \kappa_{\bfx_1}}{\tau_{\bfx_1}}
 = \sgn(\lambda)\cdot \mu_1. \qedhere\]
\end{proof}

This proposition provides a necessary condition for similarity of helical curves. The following example shows that the converse does not hold in general.
\begin{example}\label{quint}
The helical quintics $\CCC_1,\CCC_2$ parametrized by
\begin{align*}
\bfx_1(t) & = \left(\frac{3}{4}t^5 + \frac{3}{8}t^4+\frac{1}{4}t^3, \frac{4}{5}t^5+t^4\ , -\frac{3}{5}t^5+\frac{1}{2}t^4+\frac{1}{3}t^3 \right),\\
\bfx_2(t) & = \left(\frac{3}{2}t^5 + \frac{3}{4}t^4+t^3\ \ , \frac{6}{5}t^5+3t^4, -\frac{8}{5}t^5+t^4+\frac{4}{3}t^3\ \  \right).
\end{align*}
have proportionality constants $\mu_1=\mu_2=-4/3$. However, by applying Algorithm {\tt Similar3D} in Section \ref{sec:detecting}, one can show that $\CCC_1$ and $\CCC_2$ are not similar.
\end{example}

In order to check whether the condition in Proposition \ref{prop:helicalequalmu} is sufficient, we tried first several examples of helical cubics, following the method for constructing these curves presented in \cite{Farouki}. Interestingly, we could not find any counterexample with helical cubics, leaving us to conjecture that the converse of Proposition~\ref{prop:helicalequalmu} holds for helical cubics.

\begin{conjecture}
Any two cubic rational helical space curves with proportionality constants of equal modulus are similar.
\end{conjecture}

Helical rational curves with nonzero proportionality constants do exist. See \cite[\S 23]{Farouki} and \cite{Fatma} for more examples and properties that allow to construct rational curves of this type.

\section{Detecting and finding similarities of rational space curves}\label{sec:detecting}
\noindent Let $\CCC_1,\CCC_2$ be curves with parametrizations $\bfx_1, \bfx_2$ as in \eqref{eq:parametrizations}. In this section we first present a criterion for whether $f(\CCC_1) = \CCC_2$ for a similarity $f$ with a given ratio $\lambda_0$. Next, to determine the potential ratios $\lambda_0$, we develop separate methods for helical and non-helical curves. The section concludes with a method for finding the similarities with a given ratio $\lambda_0$.

We will use the following standard notions for multivariate polynomials $p\in \RR[x_1, \ldots, x_n]$, viewed as a polynomial in $x_n$ with coefficients in $\RR[x_1,\ldots, x_{n-1}]$. The \emph{leading term} of $p$ with respect to $x_n$ is the monomial of $p$ with highest degree in $x_n$, and its coefficient is called the \emph{leading coefficient}. Moreover, the \emph{content} of $p$ with respect to $x_n$ is the greatest common divisor of its coefficients, viewed as elements of $\RR[x_1,\ldots,x_{n-1}]$.

\subsection{A criterion and algorithm for detecting similarity}
\noindent Since $\kappa_{\bfx_i}^2$ and $\tau_{\bfx_i}$, with $i = 1,2$, are rational, we can write
\[\kappa_{\bfx_i}^2(t) =: \frac{A_i(t)}{B_i(t)}, \qquad \tau_{\bfx_i}(t) =: \frac{C_i(t)}{D_i(t)}, \qquad i = 1,2,\]
for coprime pairs $(A_i, B_i)$ and $(C_i, D_i)$, $i = 1,2$, of polynomials. Let
\begin{equation}\label{eq:KT}
\begin{aligned}
K_\lambda(t,s) :=\ & A_1(t)B_2(s) - \lambda^2 \cdot A_2(s)B_1(t),\\
T_\lambda(t,s) :=\ & C_1(t)D_2(s) - \lambda   \cdot C_2(s)D_1(t)
\end{aligned}
\end{equation}
be the result of clearing denominators in the expressions
$\kappa_{\bfx_1}^2(t) - \lambda^2 \kappa_{\bfx_2}^2 (s) = 0$
and $\tau_{\bfx_1}(t) - \lambda \tau_{\bfx_2}(s) = 0$. Note that $K_{-\lambda} = K_{\lambda}$. For a fixed $\lambda$, we consider the bivariate greatest common divisor and $s$-resultant

\begin{equation}\label{eq:GR}
G_\lambda := \gcd(K_\lambda, T_\lambda) ,\qquad R_\lambda := \Res_s(K_\lambda, T_\lambda). 
\end{equation}

To any M\"obius transformation $\varphi$ as in \eqref{eq:Moebius}, associate the \emph{M\"obius-like} polynomial
\begin{equation}\label{eq:F}
F(t, s) := (ct + d)s - (at + b),\qquad ad - bc\neq 0,
\end{equation}
as the result of clearing denominators in $s - \varphi(t) = 0$. Note that $F$ is irreducible since $ad - bc\neq 0$. 

The following theorem provides a criterion for similarity of $\CCC_1$ and $\CCC_2$ with a given ratio.

\begin{theorem} \label{funda-sym}
Let $\bfx_1,\bfx_2$ as in \eqref{eq:parametrizations} define curves $\CCC_1, \CCC_2$. There exists a similarity $f(\tx) = \lambda_0\mQ\tx + \bfb$ such that $f(\CCC_1) = \CCC_2$ if and only if there exists a polynomial $F$ of type \eqref{eq:F} dividing $G_{\lambda_0}$, associated with a M\"obius transformation $\varphi$ satisfying \eqref{eq:arc} with $\lambda = \lambda_0$.
\end{theorem}

\begin{proof}
``$\Longrightarrow$'': If $f(\CCC_1) = \CCC_2$ for some similarity $f(\tx)=\lambda_0\mQ\tx+\bfb$, by Theorem~\ref{th-fund} there exists a M\"obius 
transformation $\varphi$ such that $f \circ \bfx_1 = \bfx_2 \circ \varphi$. Let $F$ be the M\"obius-like polynomial associated with $\varphi$. The points $(t,s)$ for which $K_{\lambda_0}(t,s) = T_{\lambda_0}(t,s) = 0$ are the points satisfying $\kappa_{\bfx_1}(t) = |\lambda_0| \kappa_{\bfx_2} (s)$ and $\tau_{\bfx_1}(t) = \lambda_0 \tau_{\bfx_2}(s)$. By Lemma \ref{eq:curvaturetorsionsimilar}, this includes the zero set $\{(t,s) : s = \varphi(t)\}$ of $F(t,s)$. Since $F$ is irreducible, B\'ezout's theorem implies that $F$ divides $K_{\lambda_0}$ and $T_{\lambda_0}$, and therefore $G_{\lambda_0}$ as well. Moreover, since $\mQ$ is orthogonal,
\[ \|(\bfx_2 \circ \varphi)'\| = \|(f\circ \bfx_1)'\| = \| \lambda_0 \mQ \bfx_1'\| = |\lambda_0|\cdot \|\bfx_1'\|. \]

``$\Longleftarrow$'':
Let $\varphi$ be the transformation associated to $F$. Let $t_0\in I\subset \RR$ be such that $\bfx_1(t)$ is a regular point on $\CCC_1$ for every $t\in I$, and consider the arc length function
\[
s = s(t) := \int_{t_0}^t \| \bfx_1'(t)\|\rmd t,\qquad t\in I, \]
which (locally) has an infinitely differentiable inverse $t = t(s)$. For $\tbfx_2 := \lambda_0^{-1}\bfx_2$,
\[\left\|\frac{\rmd}{\rmd s} (\bfx_1 \circ t)\right\|
= \left\|\frac{\rmd \bfx_1}{\rmd t} \frac{\rmd t}{\rmd s}\right\|
= 1
= \frac{1}{|\lambda_0|}\left\|\frac{\rmd }{\rmd t} (\bfx_2 \circ \varphi) \frac{\rmd t}{\rmd s}\right\|
= \left\|\frac{\rmd }{\rmd s} (\tbfx_2 \circ \varphi \circ t)\right\|,
\]
by \eqref{eq:arc}, so that both $\bfx_1\circ t$ and $\tbfx_2 \circ \varphi \circ t$ are parametrized by arc length.

Since $F$ divides $G_{\lambda_0}$, any zero $\big(t, \varphi(t)\big)$ of $F$ is also a zero of $K_{\lambda_0}$ and $T_{\lambda_0}$, implying that $\kappa_{\bfx_1} = |\lambda_0|\cdot \kappa_{\bfx_2} \circ \varphi$ and $\tau_{\bfx_1} = \lambda_0 \tau_{\bfx_2} \circ \varphi$. Together with Lemmas~\ref{lem:kappatauf} and \ref{lem:kappatauphi}, this yields
\[ \kappa_{\bfx_1\circ t}
 = \kappa_{\bfx_1} \circ t
 = |\lambda_0|\cdot \kappa_{\bfx_2} \circ \varphi \circ t 
 = \kappa_{\tbfx_2} \circ \varphi \circ t 
 = \kappa_{\tbfx_2\circ \varphi \circ t}, \]
\[ \tau_{\bfx_1\circ t}
 = \tau_{\bfx_1} \circ t
 = \lambda_0\cdot \tau_{\bfx_2} \circ \varphi \circ t
 = \tau_{\tbfx_2} \circ \varphi \circ t
 = \tau_{\tbfx_2 \circ \varphi \circ t}. \]
The fundamental theorem of space curves \cite[\S 1--5]{Docarmo} then implies that there exists an isometry $\tilde{f}(\tx) = \mQ \tx + \bfb$, with $\det(\mQ) = 1$, such that $\tilde{f} \circ \bfx_1\circ t = \tbfx_2\circ \varphi\circ t$ on $s(I)$. In terms of the similarity $f(\tx) := \lambda_0 \tilde{f}(\tx)$, it follows that
\[ f\big(\bfx_1(t)\big) = \lambda_0 \tilde{f} \big(\bfx_1(t)\big) = \lambda_0\tbfx_2\big(\varphi(t)\big) = \bfx_2\big(\varphi(t)\big),\qquad t\in I. \]
Therefore the irreducible algebraic curves $f(\CCC_1)$ and $\CCC_2$ have infinitely many points in common, implying $f(\CCC_1) = \CCC_2$.
\end{proof} 

If some tentative values for $\lambda_0$ are known, similarity of curves can be quickly detected with this criterion, by checking if $G_{\lambda_0}$ has some M\"obius-like factor. In order to do this, taking into account that $\lambda_0$ might be an algebraic number, we can use techniques for factoring bivariate polynomials with coefficients in an algebraic number field. For instance, the command {\tt AFactor} in Maple 18 is fast and efficient. To illustrate this, it computes the factorization
\[
\left(s^2\sqrt{2}t-s^2t^2-\frac{1}{2}s^2+t^2\right)\cdot \left(st-\frac{1}{3}\sqrt{3}\right)\cdot \left(st+\frac{1}{3}\sqrt{3}\right)\cdot p(t,s),
\]
where $p(t,s)$ is a dense polynomial in $t,s$ of total degree 18, in 0.109 seconds using the machine described in Section \ref{sec:experimentation}. By Theorem \ref{funda-sym}, whenever the associated M\"obius transformation satisfies \eqref{eq:arc}, the existence of such a factor is equivalent to $\CCC_1$ and $\CCC_2$ being similar. 

Thus we arrive at Algorithm {\tt Similar3D} for checking whether $\CCC_1$ and $\CCC_2$ are similar. Note that by Theorem \ref{preserve}, any `self-similarity' of an irreducible algebraic space curve that is not a line is a symmetry. Therefore, for $\bfx_1 = \bfx_2$ one has $|\lambda| = 1$, and Algorithm \texttt{Similar3D} reduces to the algorithm presented in \cite{AHM15} for detecting symmetries of algebraic space curves.

\begin{algorithm}[h!]
\begin{algorithmic}[1]
\REQUIRE Two proper parametrizations $\bfx_1,\bfx_2$ of two space curves $\CCC_1,\CCC_2$.
\ENSURE Whether there exists a similarity $f(\tx) = \lambda\mQ\tx + \bfb$ with $f(\CCC_1) = \CCC_2$.
\STATE If $\CCC_1$ and $\CCC_2$ are both lines or both circles, return {\tt TRUE}. Otherwise:
\STATE If $\CCC_1$ or $\CCC_2$ is a circle or a line, return {\tt FALSE}.
\STATE Find the curvatures $\kappa_{\bfx_1},\kappa_{\bfx_2}$ and torsions $\tau_{\bfx_1},\tau_{\bfx_2}$ from \eqref{kt}. 
\STATE Find the polynomials $K_\lambda$ and $T_\lambda$ from \eqref{eq:KT}.
\STATE Find $\mu_1 := \kappa_{\bfx_1}/\tau_{\bfx_1}$ and $\mu_2 := \kappa_{\bfx_2}/\tau_{\bfx_2}$.
\STATE If only one among $\mu_1,\mu_2$ is constant, return {\tt FALSE}.
\STATE If $\mu_1, \mu_2$ are both constant (\emph{helical case}):
\begin{itemize}
\item[\footnotesize{7.1}]
If $|\mu_1|\neq|\mu_2|$  return {\tt False}. Otherwise:
\item[\footnotesize{7.2}] Let $G_{\lambda} := T_\lambda$.
\item[\footnotesize{7.3}] Choose $t_0\in \QQ$ such that the evaluation at $t=t_0$ of the leading coefficient of $G_{\lambda}(t,s)$ with respect to $s$ is not identically zero.
\item[\footnotesize{7.4}] Find the sets $\SSS_0, \SSS_1, \SSS_2$ of tentative $\lambda$ using the method in Section \ref{sec:detecthelical}.
\item[\footnotesize{7.5}] For each $\lambda \in \SSS_0\cup \SSS_1\cup \SSS_2$, check whether $G_\lambda$ contains a M\"obius-like factor $F$ for which the associated M\"obius transformation $\varphi$ satisfies \eqref{eq:arc}.\!\!\!\!
\item[\footnotesize{7.6}] If some $\lambda$ succeeds, return {\tt True}, otherwise return {\tt False}.
\end{itemize}
\STATE If $\mu_1, \mu_2$ are not constant (\emph{non-helical case}):
\begin{itemize}
\item[\footnotesize{8.1}] Find the resultant $R_\lambda = \Res_s (K_\lambda, T_\lambda)$.
\item[\footnotesize{8.2}] Find the set $\SSS_0$ of tentative $\lambda$ using the method in Section \ref{method-non-helical}.
\item[\footnotesize{8.3}] For each $\lambda\in \SSS_0$, check whether $G_\lambda$ contains a M\"obius-like factor $F$ for which the associated M\"obius transformation $\varphi$ satisfies \eqref{eq:arc}.
\item[\footnotesize{8.4}] In the affirmative case, return {\tt True}, otherwise return {\tt False}.
\end{itemize}
\end{algorithmic}
\caption*{{\bf Algorithm} {\tt Similar3D}}
\end{algorithm}

It remains to compute the sets $\SSS_0, \SSS_1, \SSS_2$ of tentative values for $\lambda_0$ in the next two sections, where it is necessary to distinguish between helical and non-helical curves.

\subsection{Finding the ratio for non-helical curves}\label{method-non-helical}
\noindent Assume $\bfx_1,\bfx_2$ define non-helical curves. By the following proposition, there are only finitely many nonzero $\lambda$ for which the resultant $R_\lambda$ is identically zero.

\begin{proposition} \label{prop:why}
The resultant $R_\lambda$ is identically zero if and only if $\CCC_1$, $\CCC_2$ are helical curves with proportionality constants $\mu_1,\mu_2$ satisfying $|\mu_1|=|\mu_2|$.
\end{proposition}

\begin{proof} ``$\Longleftarrow$'': Since the proportionality constants have the same absolute value $\mu:=|\mu_1|=|\mu_2|$,
\[ \frac{A_i(t)}{B_i(t)} = \kappa_{{\bfx}_i}^2(t) = \mu^2\cdot \tau_{\bfx_i}^2(t) = \mu^2\cdot \frac{C_i^2(t)}{D_i^2(t)}, \]
with $\mu\neq 0$ because of Lemma \ref{lem:nonzeromu}. Therefore
\begin{align*}
K_{\lambda}(t,s) & = \mu^2\cdot \big(C_1^2(t)D_2^2(s) - \lambda^2C_2^2(s)D_1^2(t)\big)\\
                 & = \mu^2\cdot \big(C_1(t)D_2(s) - \lambda C_2(s)D_1(t)\big) \cdot \big(C_1(t)D_2(s) + \lambda C_2(s)D_1(t) \big)\\
                 & = \mu^2\cdot T_{\lambda}(t,s) \cdot T_{-\lambda}(t,s).
\end{align*}
Hence $K_\lambda$ has a non-trivial factor, depending on $s$, in common with both $T_{\lambda}$ or $T_{-\lambda}$, since $K_\lambda = K_{-\lambda}$. It follows that $R_\lambda$ is identically zero.

``$\Longrightarrow$'': If $R_\lambda$ is identically zero then $K_\lambda$, $T_\lambda$ have nontrivial greatest common divisor $G_\lambda$. Suppose $T_\lambda$ has a factor $S$ not depending on $\lambda$. Then $S$ divides both $T_\lambda$ and $T_0(t,s) = C_1(t) D_2(s)$, and therefore also $C_2(s) D_1(t)$, contradicting that $C_1,D_1$ and $C_2,D_2$ are coprime. A similar argument shows that any nonconstant factor of $K_\lambda$ depends on $\lambda$. It follows that $G_\lambda$ is a linear polynomial in $\lambda$ in constant proportion with $T_\lambda$. 
Since $G_\lambda, G_{-\lambda}$ both divide $K_\lambda = K_{-\lambda}$, which is a quadratic polynomial in $\lambda$, it follows that
\[ K_{\lambda} = \nu \cdot T_\lambda \cdot T_{-\lambda}\]
for some nonzero constant $\nu$. Comparing coefficients it follows that 
\[ A_1(t) B_2(s) = \nu \cdot C_1^2(t) D_2^2(s), \qquad A_2(s) B_1(t) = \nu \cdot C_2^2(s) D_1^2(t). \]
Dividing these equations yields
\[ \frac{\kappa_{\bfx_1}^2(t)}{\kappa_{\bfx_2}^2(s)} = \frac{A_1(t) B_2(s)}{B_1(t)A_2(s)} = \frac{C_1^2(t) D_2^2(s)}{D_1^2(t) C_2^2(s)} = \frac{\tau_{\bfx_1}^2(t)}{\tau_{\bfx_2}^2(s)}, \]
or equivalently
\[ \frac{\kappa_{\bfx_1}^2(t)}{\tau_{\bfx_1}^2(t)} = \frac{\kappa_{\bfx_2}^2(s)}{\tau_{\bfx_2}^2(s)}, \]
which must be constant. After taking square roots the statement follows.
\end{proof}

Let $\Lambda^{\star}(\lambda)$ be the content of the resultant $\mbox{Res}_s(K_{\lambda},T_{\lambda})$, viewed as a polynomial in $t$ with coefficients depending on $\lambda$. Let $\mbox{lc}_s(K_{\lambda})$, $\mbox{lc}_s(T_{\lambda})$ be the leading coefficients with respect to $s$ of $K_{\lambda},T_{\lambda}$. 
Notice that whenever $\mbox{lc}_s(K_{\lambda})$, $\mbox{lc}_s(T_{\lambda})$ do not vanish identically and simultaneously for $\lambda = \lambda_0$, then $R_{\lambda_0}$ is the result of specializing $\mbox{Res}_s(K_{\lambda},T_{\lambda})$ at $\lambda=\lambda_0$ (see Lemma 4.3.1 of \cite{winkler}). Let $\Lambda(\lambda)$ be the product of $\Lambda^{\star}(\lambda)$ and the content of $\gcd(\mbox{lc}_s(K_{\lambda}),\mbox{lc}_s(T_{\lambda}))$ with respect to $t$. Let $\SSS_0$ be the nonzero real roots of $\Lambda$.

\begin{proposition}\label{howto}
Let $\bfx_1, \bfx_2$ define non-helical curves $\CCC_1, \CCC_2$ satisfying $f(\CCC_1) = \CCC_2$ for some similarity $f(\tx) = \lambda_0\mQ\tx + \bfb$. Then $\lambda_0\in \SSS_0$.
\end{proposition}

\begin{proof}
By Lemma \ref{lem:curvaturetorsionsimilar}, there exists a M\"obius transformation $\varphi$ such that
\[ K_{\lambda_0} \big(t, \varphi(t)\big) = T_{\lambda_0} \big(t, \varphi(t)\big) = 0,\ 
\text{and therefore }G_{\lambda_0}\big(t, \varphi(t)\big) = 0,\]
hold identically. Hence the M\"obius-like polynomial $F$ associated to $\varphi$ divides $G_{\lambda_0}$. Since the polynomial ring $\RR[t]$ is an integral domain, the bivariate polynomial $G_{\lambda_0}$ is non-constant precisely when the resultant $R_{\lambda_0}$ is identically zero, which implies $\Lambda(\lambda_0) = 0$.
\end{proof}

Proposition \ref{howto} provides tentative values of $\lambda_0$, which must be tested afterwards using Theorem \ref{funda-sym}.

\begin{example}\label{ex-crunodes}
Let $\CCC_1$ be the crunode parametrized by 
\[ \bfx_1(t) = \left(\frac{t}{t^4 + 1}, \frac{t^2}{t^4 + 1}, \frac{t^3}{t^4 + 1} \right). \]
Based on the reparametrization $\phi(t) = t+1$ and similarity
\begin{equation}\label{example2}
f(\tx)=\lambda\mQ \tx+\bfb,\qquad
\lambda = 2,\qquad
\mQ =
\begin{bmatrix}
\phantom{+}3/5 & 4/5 & 0\\
-4/5 & 3/5 & 0\\
 \phantom{+}0 & 0 & 1
\end{bmatrix},\qquad 
\bfb =
\begin{bmatrix} 0\\ 0\\ 2 \end{bmatrix},
\end{equation}
we define another crunode $\CCC_2 := f(\CCC_1)$ parametrized by $\bfx_2 = f\circ \bfx_1 \circ \phi$, i.e.,
\[
\bfx_2(t) = \left( 
\frac25 \frac{(t+1)(4t+7)}{(t+1)^4+1},
\frac25 \frac{(t+1)(3t-1)}{(t+1)^4+1},
        \frac{2(t+1)^3}{(t+1)^4+1} + 2\right).\]
The curves $\CCC_1$ and $\CCC_2$ are shown in Figure \ref{fig:crunodes}, together with the invariant sets of their symmetries, i.e., two planes of reflection and an axis of rotation.

One verifies that $\kappa^2_{\bfx_1}, \kappa^2_{\bfx_2}$ are rational functions where the numerators and denominators have degree $36$. Furthermore, the numerator and denominator of $\tau_{\bfx_1}, \tau_{\bfx_2}$ have degree 8. Therefore $K_\lambda(t,s)$ has bidegree $(36,36)$ and $T_\lambda(t,s)$ has bidegree $(8,8)$. After computing $R_\lambda = \Res_s(K_{\lambda},T_{\lambda})$, we get $\Lambda(\lambda)=\Lambda^{\star}(\lambda)=\lambda^2-4$, so $\SSS_0 = \{-2,2\}$ contains the tentative values for $\lambda$. 
For $\lambda_0=2$ we obtain $G_2(t,s) = (s+t+1)(s-t+1)$ and two corresponding M\"obius transformations satisfying \eqref{eq:arc}, namely $\varphi_1(t) = - t - 1$ and $\varphi_2(t) = t - 1$. For $\lambda_0 = -2$ we obtain $G_{-2}(t,s)=(st+t+1)(st+t-1)$ and two corresponding M\"obius transformations satisfying \eqref{eq:arc}, namely $\varphi_3(t)=-(t+1)/t$ and $\varphi_4(t)=(-t+1)/t$.
Therefore ${\mathcal C}_1$ and ${\mathcal C}_2$ are similar, and there are four different similarities mapping one to the other. 
\end{example} 

\begin{figure}
\begin{center}
\subfloat[]{\includegraphics[scale=0.18, clip = true, trim = 75 215 80 160]{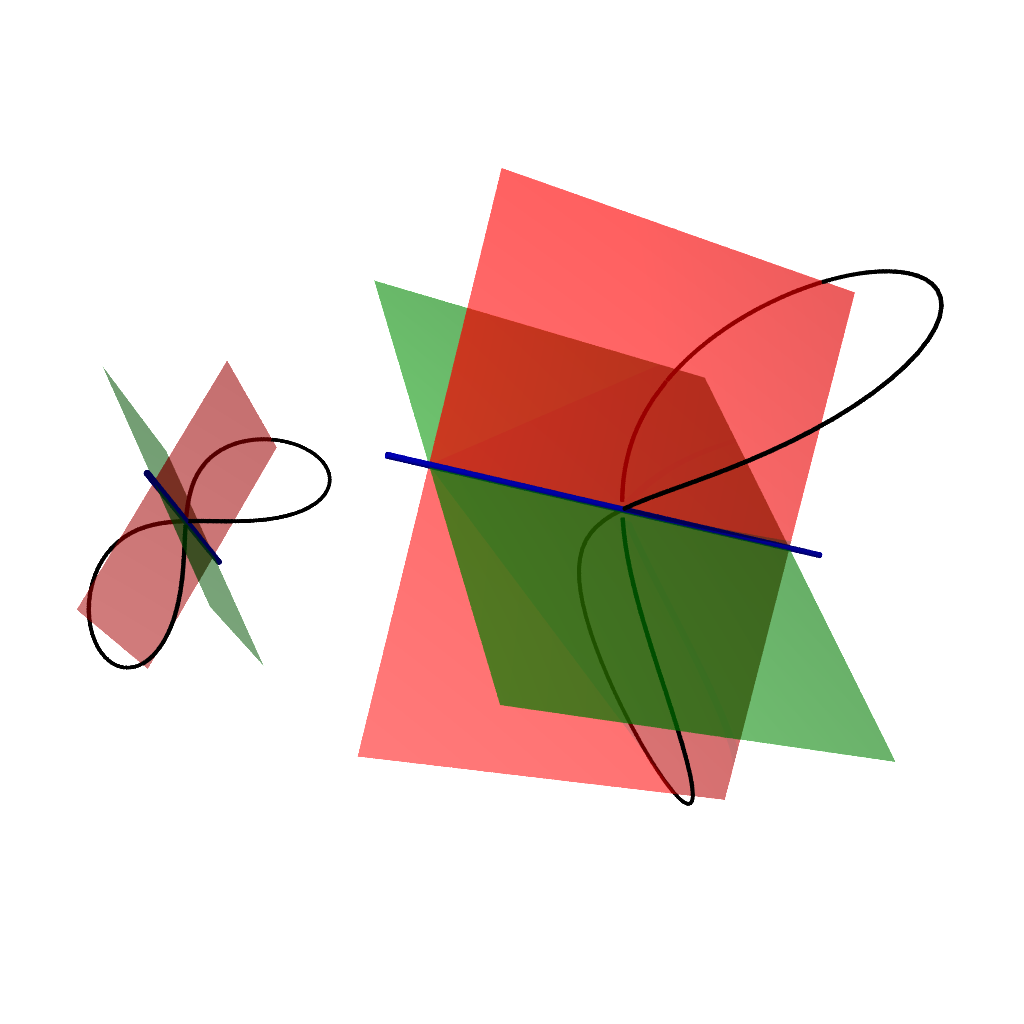}\label{fig:crunodes}}\qquad
\subfloat[]{\includegraphics[scale=0.18, clip = true, trim = 30 270 70 290]{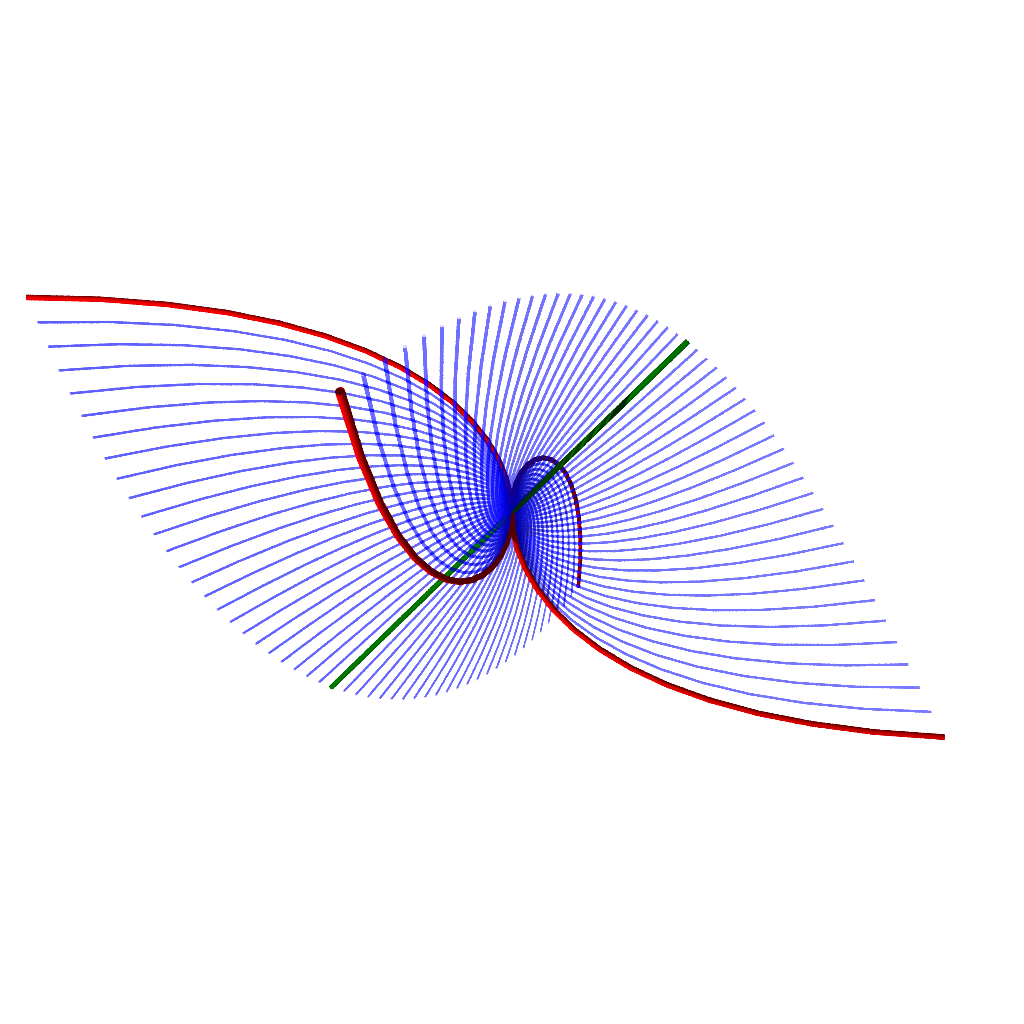}\label{fig:helicalcurves}}
\end{center}
\caption{Left: The similar crunode curves from Example \ref{ex-crunodes}. Right: The family of helical curves $\CCC_\alpha$, with $-1\leq \alpha\leq 1$, from Example \ref{ex:HelicalCurves}, with the curves $\CCC_{-1}, \CCC_0, \CCC_1$ emphasized.}\label{fig:spacecurves}
\end{figure}

\subsection{Finding the ratio for helical curves}\label{sec:detecthelical}
\noindent Assume that $\CCC_1, \CCC_2$ are similar helical curves. Then their proportionality ratios are equal up to a sign by Proposition \ref{prop:helicalequalmu}, and $\gcd(K_{\lambda},T_{\lambda})=T_{\lambda}$ for any $\lambda$ by the proof of Proposition \ref{prop:why}. Therefore $R_\lambda \equiv 0$, and we cannot use the method in Section \ref{method-non-helical} to find the potential ratios $\lambda$. However, since $G_{\lambda} = T_{\lambda}$ is known, we can directly apply Theorem \ref{funda-sym} to find the $\lambda$ values for which $G_{\lambda}$ has a M\"obius-like factor. In order to do this, we adapt the method in \cite[\S 3.2]{AHM15}, where the problem of directly computing the M\"obius-like factors of a bivariate polynomial is solved. The idea is that, if $\lambda_0$ is the ratio we are seeking, the M\"obius transformation $\varphi$ corresponding to the similarity is implicitly defined by $G_{\lambda_0}$, so that it can be reconstructed from its local data. 

\begin{lemma}
Let either $c = 0$ or $t_0\neq -d/c$, and consider the Taylor expansion
\begin{equation}\label{eq:MoebiusTaylor}
\varphi(t) = \frac{at + b}{ct + d} = s_0 + s_0'(t - t_0) + \frac12 s_0'' (t - t_0)^2  + \cdots
\end{equation}
Then, as homogeneous coordinates,
\begin{equation}\label{eq:abcd-homogeneous}
[a:b:c:d] = [2 (s_0')^2 - s_0 s_0'' : 2s_0s_0' + t_0s_0s_0'' - 2t_0(s_0')^2 : -s_0'' : 2s_0' + t_0 s_0''].
\end{equation}
\end{lemma}

\begin{proof}
Differentiating \eqref{eq:MoebiusTaylor} and evaluating at $t = t_0$ yields
\[
s_0   = \varphi  (t_0) = \frac{at_0 + b}{ct_0 + d},\quad
s_0'  = \varphi' (t_0) = \frac{\Delta}{(ct_0 + d)^2},\quad
s_0'' = \varphi''(t_0) = \frac{-2c\Delta}{(c t_0 + d)^3}.
\]
The statement follows from a straightforward calculation.
\end{proof}

Let $t_0\in\QQ$ be such that the evaluation at $t=t_0$ of the leading coefficient $\mathrm{lc}_s(G_{\lambda}(t,s))$ of the polynomial $G_{\lambda}(t,s)$ with respect to $s$ is not identically zero. Let $L(\lambda)$ be $\mathrm{lc}_s(G_{\lambda}(t,s))$ evaluated at $t_0$. In order to detect M\"obius-like factors of $G_\lambda$ using the implicit function theorem, we need to exclude any $\lambda$ from $\SSS_1 \cup \SSS_2$, with $\SSS_1 := \{0\neq \lambda\in \RR\,:\,L(\lambda) = 0\}$ and
\[\SSS_2 := \left\{0\neq \lambda \in \RR\,:\, G_{\lambda}(t_0,s)=0,\mbox{ }\frac{\partial G_\lambda}{\partial s} (t_0, s) = 0 \text{ for some } s\in \CC\right\}.\]
The elements of $\SSS_2$ can be found by eliminating the variable $s$ from the bivariate polynomial system $G_{\lambda}(t_0,s) = \frac{\partial G_\lambda}{\partial s} (t_0, s) = 0$ in $\lambda, s$, for instance using the Sylvester resultant.	

Suppose $\lambda \notin \SSS_1 \cup \SSS_2$. With the dependency on the variable $\lambda$ understood, write $G = G_\lambda$ and $G_t, G_s, G_{tt}, G_{ts}, G_{ss}$ for the first and second order partial derivatives of $G$. Suppose $G$ has a M\"obius-like factor $F$, and let $s_0$ be a variable required to satisfy $F(t_0, s_0) = 0$. Since $G(t_0, s_0) = 0$ and $G_s(t_0, s_0) \neq 0$ one has $\frac{\partial F}{\partial s}(t_0, s_0)\neq 0$, and the equation $F(t,s) = 0$ implicitly defines a function $s = \varphi(t)$ in a neighborhood of $t_0$ with $s_0 = \varphi(t_0)$ as in \eqref{eq:MoebiusTaylor}.

In order to determine $F$, we find expressions for $s_0', s_0''$ in terms of $s_0,\lambda$, using that $\varphi(t)$ is also implicitly defined by $G(t,s) = 0$, because $F$ divides $G$ and $G_s (t_0, s_0)\neq 0$. Differentiating once and twice the identity $G\big(t,\varphi(t)\big)=0$ with respect to $t$, solving for $\varphi', \varphi''$, and evaluating at $t_0$ expresses 
\begin{align}
s'_0  & = \varphi'(t_0) = - \frac{G_t}{G_s}(t_0, s_0),\\
s''_0  & = \varphi''(t_0) = - \frac{G_s^2 G_{tt} - 2G_tG_sG_{ts} + G_t^2 G_{ss} }{G_s^3} (t_0,s_0)
\end{align}
in terms of the unknown $s_0$. Substituting these expressions into \eqref{eq:abcd-homogeneous} and multiplying by $-G_s^3(t_0,s_0)$ yields polynomial expressions for the coefficients of $\varphi$ in terms of $s_0$,
\begin{equation}\label{eq:abcdG}
\begin{aligned}
a(s_0, \lambda) & = - \big(G_s^2G_{tt} - 2 G_t G_s G_{ts} + G_t^2 G_{ss}\big) s_0 - 2G_t^2 G_s ,\\
b(s_0, \lambda) & = + \big(G_s^2G_{tt} - 2 G_t G_s G_{ts} + G_t^2 G_{ss}\big) t_0 s_0 + 2s_0 G_t G_s^2 + 2t_0 G_t^2 G_s,\!\!\!\!\\
c(s_0, \lambda) & = - \big(G_s^2G_{tt} - 2 G_t G_s G_{ts} + G_t^2 G_{ss}\big),\\
d(s_0, \lambda) & = + \big(G_s^2G_{tt} - 2 G_t G_s G_{ts} + G_t^2 G_{ss}\big) t_0 + 2G_tG_s^2,
\end{aligned}
\end{equation}
where these expressions are understood to be evaluated at $(t_0, s_0)$.

The polynomial $F$ divides $G$ if and only if the resultant $\Res_s(F,G)$ is identically zero, or equivalently precisely when
\begin{equation} \label{eq:G}
 0 = G\big(t, \varphi(t)\big) = G\left(t, \frac{a(s_0,\lambda)t + b(s_0,\lambda)}{c(s_0,\lambda)t + d(s_0,\lambda)}\right) 
\end{equation}
holds identically. Clearing denominators yields a polynomial $P(t)$, whose coefficients are polynomials $P_i(s_0,\lambda)$. Then $F$ divides $G$ if and only if there exist $s_0\in \RR$ and $0\neq \lambda\in \RR$ for which the $P_i$ are simultaneously zero, i.e., when there is such a point $(s_0, \lambda)$ on the real variety generated by the ideal $\langle P_i \rangle_i \subset \RR[s,\lambda]$. Using Gr\"obner bases, one eliminates the variable $s$ from this ideal, resulting in a principal ideal $\langle\Lambda \rangle \subset \RR[\lambda]$. Let $\SSS_0 := \{0\neq \lambda\in \RR\,:\,\Lambda(\lambda) = 0\}$. We have shown:

\begin{theorem}
Suppose $f(\CCC_1) = \CCC_2$ for a similarity $f$ with ratio $\lambda$.
Let $t_0$ be such that $\mathrm{lc}_s(G_{\lambda}(t,s))$ does not vanish identically at $t=t_0$. Then $\lambda\in \SSS_0\cup\SSS_1\cup\SSS_2$. 
\end{theorem}

Therefore, $f(\CCC_1) = \CCC_2$ for a similarity $f$ with ratio $\lambda_0$ if and only if
\begin{enumerate}
\item[(i)]  $\lambda_0\in \SSS_1\cup\SSS_2$ and $G_{\lambda_0}$ has a M\"obius-like factor, or
\item[(ii)] $\lambda_0\in \SSS_0$ and the polynomials $P_i$, after substituting $\lambda_0$, have a common real root $s_0$,
\end{enumerate}
for which, in either case, the corresponding M\"obius transformation satisfies \eqref{eq:arc}. 

\begin{example}\label{ex:HelicalCurves}
Consider the family of curves $\{\CCC_\alpha\}_\alpha$ defined by the parametrizations
\[ \bfx_\alpha(t) = \left(-\frac{1}{3}t^3+\alpha^2t,\frac{2}{3}t^3+\alpha t^2,\frac{2}{3}t^3-\alpha t^2\right),\qquad \alpha\in \RR. \]
These curves are shown in Figure \ref{fig:helicalcurves} for parameters $-1\leq \alpha \leq 1$, with the curves $\CCC_{-1}, \CCC_0, \CCC_1$ emphasized. Except for the line $\CCC_0$, each curve $\CCC_\alpha$ is a cubic helical curve with proportionality constant $\mu_\alpha$ satisfying $|\mu_\alpha| = \sqrt{2}$, since
\[ \kappa_{\bfx_\alpha} = \frac{2|\alpha|\sqrt{2}}{(\alpha^2+3t^2)^2},\qquad \tau_{\bfx_\alpha}=\frac{2\alpha}{(\alpha^2+3t^2)^2}. \]
In order to determine if the curves $\CCC_1, \CCC_{-1}$ are similar, we compute
\[G_{\lambda}(t,s)= 9\lambda t^4 + 9s^4 + 6\lambda t^2 + 6s^2 + \lambda + 1.\]
Letting $t_0 = 1$, one has $G_{\lambda}(t_0,s) = 9s^4 + 6s^2 + 16\lambda + 1$ with constant leading coefficient $L(\lambda) = 9$, implying $\SSS_1 = \emptyset$. Moreover, $\frac{\partial G_{\lambda}}{\partial s}(t_0,s) = 36s^3 + 12s$, so $\SSS_2=\{-1/16\}$. Since
\[ s'_0  = \frac{-4\lambda}{s_0(3s_0^2+1)},\qquad
   s''_0 = \frac{-2\lambda (45s_0^6+30s_0^4+72\lambda s_0^2+5s_0^2 + 8\lambda)}{s_0^3(3s_0^2+1)^3}, \]
we have, after scaling by a common factor,
\begin{equation}\label{eq:exhel}
\begin{aligned}
a(s_0, \lambda) & = -s_0(45s_0^6 + 30s_0^4 + 120\lambda s_0^2 + 5s_0^2 + 24\lambda),\\
b(s_0, \lambda) & = 3s_0(27s_0^6 + 18s_0^4 +  40\lambda s_0^2 + 3s_0^2 +  8\lambda),\\
c(s_0, \lambda) & = -(45s_0^6 + 30s_0^4 + 72\lambda s_0^2 + 5 s_0^2 + 8\lambda),\\
d(s_0, \lambda) & = 81s_0^6 + 54s_0^4 + 72\lambda s_0^2 + 9 s_0^2 + 8\lambda.
\end{aligned}
\end{equation}
Substituting \eqref{eq:exhel} into \eqref{eq:G} and clearing denominators, we get a polynomial $P(t)$ whose coefficients are polynomials $P_i(s_0,\lambda)$. Eliminating the variable $s_0$ from the ideal $\langle P_i \rangle_i$, we obtain the generator $\Lambda(\lambda) = \lambda^5(\lambda + 1)$ and $\SSS_0 = \{-1\}$. Then 
\[G_{-1}(t,s) = 3(s - t)(s + t)(3s^2 + 3t^2 + 2), \]
and the corresponding M\"obius transformations $\varphi_1(t) = t$ and $\varphi_2(t) = -t$ satisfy \eqref{eq:arc}. Since $\lambda=-1$ succeeds, one does not need to try $\lambda=-1/16 \in \SSS_2$ by Proposition \ref{lambda-unique}. We conclude that $\CCC_1$ and $\CCC_2$ are similar under two similarities. 
\end{example}

\subsection{Finding the similarities} \label{sec:finding}
\noindent Suppose that using Algorithm {\tt Similar3D} we have determined that $f(\CCC_1) = \CCC_2$ for a similarity $f(\tx)=\lambda_0\mQ\tx + \bfb$. Then we have computed the associated M\"obius transformation $\varphi$ and the ratio $\lambda_0$, and we would like to find $\mQ$ and $\bfb$. For this purpose, we adapt to our problem the discussion in \cite[\S 4]{AHM15}. By Theorem~\ref{th-fund}, 
\begin{equation} \label{fund-equality}
\lambda_0 \cdot \mQ \bfx_1(t)+\bfb=\bfx_2\big(\varphi(t)\big).
\end{equation}
Once $\mQ$ is determined, one finds $\bfb$ by evaluating \eqref{fund-equality} at $t = 0$.

Without loss of generality, we assume that $\bfx_1(t)$, and therefore any of its derivatives, is well defined at $t=0$ and that $\bfx_1'(0),\bfx_1''(0)$ are well defined, nonzero, and not parallel. This is equivalent to requiring that the curvature $\kappa_{\bfx_1}(0)$ is well defined and nonzero, which can always be achieved by a reparametrization of type $t \longmapsto t+\alpha$. 

To determine $\mQ$, we consider separately the cases when the coefficient $d$ of the M\"obius transformation $\varphi$ satisfies $d\neq 0$ or $d=0$. If $d=0$, then $0\neq \Delta = -bc$ implies $c\neq 0$, and Equation \eqref{fund-equality} becomes
\[ \lambda_0 \cdot \mQ\bfx_1(t) + \bfb = \bfx_2\big(\varphi(t)\big) = \bfx_2\big(\ta / t + \tb \big),\qquad
\ta := \frac{b}{c},\qquad \tb := \frac{a}{c}.\]
Writing $\tbfx_2(t) := \bfx_2(1/t)$, we obtain
\begin{equation} \label{spec}
 \lambda_0 \cdot \mQ\bfx_1(t) + \bfb =  \tbfx_2\big(\ta t + \tb\big).
\end{equation}
Evaluating \eqref{spec} at $t = 0$ yields
\begin{equation}\label{eq:evaldzero}
\lambda_0\cdot \mQ\bfx_1(0) + \bfb = \tbfx_2(\tb),
\end{equation}
while differentiating \eqref{spec} once and twice and evaluating at $t=0$ yields
\begin{equation}\label{eq:diff12dzero}
\lambda_0\cdot \mQ\bfx_1'(0) =  \tbfx_2'(\tb)\cdot \ta,\qquad \lambda_0\cdot \mQ\bfx_1''(0) =  \tbfx_2''(\tb)\cdot \ta^2.
\end{equation}
Taking the cross product in \eqref{eq:diff12dzero} and using \eqref{eq:orthogonalcrossproduct} with $\det(\mQ) = 1$ and $\mM = \mQ$ orthogonal,
\begin{equation}\label{eq:diff1xdiff2dzero}
\lambda_0^2 \cdot \mQ\big( \bfx_1'(0)\times \bfx_1''(0) \big) = \tbfx_2'(\tb) \times \tbfx_2''(\tb) \cdot \ta^3.
\end{equation}
Combining \eqref{eq:diff12dzero} and \eqref{eq:diff1xdiff2dzero}, with
\begin{equation}\label{eq:B}
\mB :=[\lambda_0\cdot \bfx_1'(0), \lambda_0\cdot \bfx_1''(0), \lambda_0^2\cdot \bfx_1'(0) \times \bfx_1''(0)],
\end{equation}
yields
\begin{equation}\label{eq:C}
\mQ\mB = \mC := \big[\tbfx_2'(\tb)\cdot \ta, \tbfx_2''(\tb)\cdot \ta^2, \tbfx_2'(\tb) \times \tbfx_2''(\tb) \cdot \ta^3\big],
\end{equation}
and $\mQ = \mC\mB^{-1}$.

Now let us address the case $d\neq 0$. Differentiating \eqref{fund-equality} once and twice, yields
\begin{align}
\lambda_0 \cdot \mQ\bfx_1'(t)  & = \bfx_2'\big(\varphi(t)\big)\cdot\varphi'(t) = \bfx_2'\left(\frac{a t + b}{c t + d}\right)\frac{\Delta}{(c t + d)^2}, \label{eq:first}\\
\lambda_0 \cdot \mQ\bfx_1''(t) & = \bfx_2''\big(\varphi(t)\big)\big(\varphi'(t)\big)^2 + \bfx_2'\big(\varphi(t)\big)\varphi''(t) \label{eq:second}\\
           & =\displaystyle{\bfx_2''\left(\frac{a t + b }{c t + d }\right) \frac{\Delta^2}{(c t + d )^4} - 2\bfx_2'\left(\frac{a t + b}{c t + d}\right) \frac{c\cdot \Delta}{(c t + d)^3}}, \notag
\end{align}
where $\Delta=ad-bc$. Evaluating \eqref{eq:first} and \eqref{eq:second} at $t = 0$ yields
\begin{align}
\lambda_0 \cdot \mQ \bfx_1'(0)  & = \bfx_2'(b/d)\cdot \Delta/d^2, \label{eq:diff1}\\
\lambda_0 \cdot \mQ \bfx_1''(0) & = \bfx_2''(b/d)\cdot \Delta^2/d^4 - 2 \bfx_2'(b/d)\cdot c\cdot \Delta/d^3. \label{eq:diff2}
\end{align}
Taking the cross product and using \eqref{eq:orthogonalcrossproduct} with $\mM = \mQ$ orthogonal yields
\begin{equation}\label{eq:diff1xdiff2}
\lambda_0^2 \cdot \mQ\big(\bfx_1'(0) \times \bfx_1''(0) \big)  = (\Delta^3/d^6)\cdot \big(\bfx_2'(b/d) \times \bfx_2''(b/d) \big).
\end{equation}
Since $\lambda_0$ and $\varphi$ are known, the matrix $\mQ$ can again be determined from its action on $\bfx_1'(0), \bfx_1''(0)$, and $\bfx_1'(0) \times \bfx_1''(0)$, which is given by Equations \eqref{eq:diff1}--\eqref{eq:diff1xdiff2}.

\begin{example}
Let us find the similarity between the crunode curves $\CCC_1,\CCC_2$ in Example~\ref{ex-crunodes}, corresponding to $\lambda_0=2$, $\varphi(t)=t-1$. Then $\varphi$ has coefficients $a=1$, $b=-1$, $c=0$, $d=1$, and therefore $\Delta=1$. From \eqref{eq:B}, \eqref{eq:C} one obtains 
\[
\mB = \begin{bmatrix}  2 & 0 & 0\\ 0  & 4 & 0\\ 0 & 0 &  8 \end{bmatrix},\quad
\mC = \begin{bmatrix}  6/5 & 16/5 & 0\\ -8/5  & 12/5 & 0\\ 0 & 0 & 8 \end{bmatrix},\quad
\mQ = \mC\mB^{-1} =
\begin{bmatrix} 3/5 & 4/5 & 0\\ -4/5 & 3/5 & 0\\ 0 & 0 & 1 \end{bmatrix}
\]
Substituting $t=0$ in \eqref{fund-equality} yields $\bfb=\bfx_2(-1)- 2\mQ\bfx_1(0)=[0,0,2]^\rmT$, consistent with Example~\ref{ex-crunodes}.
\end{example}

\subsection{An alternative method} 
\noindent Suppose that $\bfx_1,\bfx_2$ as in \eqref{eq:parametrizations} define curves $\CCC_1,\CCC_2$ related by a similarity $f$ corresponding to a M\"obius transformation $\varphi$ with associated M\"obius-like polynomial~$F$. Setting \eqref{eq:KT} to zero and eliminating $\lambda$, it follows that $F$ must be a factor of the polynomial 
\[ H(t,s) := A_1(t)B_2(s)C_2^2(s)D_1^2(t) - A_2(s)B_1(t)C_1^2(t)D_2^2(s). \]

Let us consider the non-helical case, for which the polynomial $H$ is not identically zero. As an alternative to the method presented in this section, one could first compute the M\"obius-like factors $F$ of $H$, and then find the similarity ratio $\lambda_0$, if it exists, as the (constant) quotient
\[\frac{C_1(t)\cdot D_2\big(\varphi(t)\big)}{C_2\big(\varphi(t)\big)\cdot D_1(t)}.\]
One can then apply Theorem \ref{funda-sym} whenever its condition holds for the pair $(F,\lambda_0)$.

The advantage of this strategy is that it avoids the resultant computation to find $\lambda_0$. However, since this computation is replaced by the factorization of the high-degree polynomial $H$ over the real (irrational) numbers, the computation time is not necessarily better. In fact, as the degree of the curves grows, numerical tests indicate that the performance is worse than the method presented before. The reason seems to be that, in contrast to the degree of $H$, the degree of $G_{\lambda_0}$ tends to stay low, as it is the result of a gcd computation. However, for curves of low degree the method presented in this subsection is a simple and viable approach to detecting similarities.

\section{Experimentation and practical performance}\label{sec:experimentation}
\noindent Algorithm {\tt Similar3D} was implemented in the computer algebra system Maple 18, and was tested on an Intel Core i7 laptop, with 2.9 GHz processor and 8~GB RAM. In this section we present tables with timings corresponding to different groups of examples.

\subsection{Random rational non-helical curves}
\noindent For similar non-helical rational curves with ratio $\lambda \in \QQ$, the bottleneck of Algorithm {\tt Similar3D} is the computation of the resultant $R_\lambda = \Res_s (K_\lambda, T_\lambda)$. In fact, we avoided the direct computation of this resultant. Instead, we computed for various values of $t_0$ the specialized \emph{bivariate} resultants $\Res_s (K_\lambda(s,t_0),$ $T_\lambda(s,t_0))$ and computed their greatest common divisor; this yields a finite list of tentative values of $\lambda$. However, even the computation of these bivariate resultants is time-consuming as the bitsizes of the coefficients or degrees grow. 

Table \ref{tab:sizedegree} lists timings for random rational non-helical parametrizations with various degrees $m$ and coefficients with bitsizes at most $\tau$. The degree of the parametrization corresponds to the highest degree in the numerators and denominators of the components. Similarly, the bitsize of the parametrization corresponds to the largest bitsize of the coefficients of the numerators and denominators of the components. In order to generate these examples, we randomly created curves $\CCC$ with given degree $m$ and bitsize $\tau$, and we ran the algorithm with $\CCC_1\equiv \CCC$ and $\CCC_2\equiv f(\CCC)$, with $f$ the similarity of \eqref{example2}. We observed that in practice almost all the time was consumed computing the tentative values of $\lambda$. 

\begin{table}[t!]
\begin{tabular*}{\columnwidth}{c @{\extracolsep{\stretch{1}}}*{4}{r}@{ }}
\toprule
CPU time &  $\tau = 4$ &  $\tau = 8$ & $\tau = 16$ & $\tau = 32$ \\ \midrule 
$m =  3$ &       0.327 &       0.375 &       0.577 &       0.842 \\ 
$m =  4$ &       0.655 &       1.170 &       1.497 &       3.120 \\ 
$m =  5$ &       1.263 &       1.700 &       3.292 &       7.098 \\ 
$m =  6$ &       1.716 &       3.900 &       6.880 &      15.288 \\ 
$m =  7$ &       4.336 &       6.896 &      14.290 &      27.659 \\ 
$m =  8$ &       8.253 &      12.683 &      22.168 &      35.927 \\ 
$m =  9$ &      11.762 &      10.998 &      21.466 &      57.424 \\ 
$m = 10$ &      12.340 &      23.509 &      46.519 &      90.746 \\ 
\bottomrule
\end{tabular*}
\caption{CPU time (seconds) for random rational parametrizations of various degrees $m$ and coefficients with bitsize bounded by $\tau$.}\label{tab:sizedegree}
\end{table}

\subsection{A family of daisies}
\noindent Table~\ref{tab:spacerose} lists timings for a family of daisies of increasing degree $m = 4j + 4$, parametrically given by
\begin{equation}\label{eq:daisies}
\hspace{-0.1em}\bfx(t) = \left(
u\sum_{i=0}^j (-1)^i {2j\choose 2i} u^{2j - 2i} v^{2i},
v\sum_{i=0}^j (-1)^i {2j\choose 2i} u^{2j - 2i} v^{2i},
\frac{1 - t^{4j + 4}}{1 + t^{4j + 4}}
\right),\hspace{-0.28em}
\end{equation}
where
\[ u = \frac{1-t^2}{1+t^2},\qquad v = \frac{2t}{1+t^2}, \qquad j = 0, 1, \ldots \]
In each case we tested Algorithm {\tt Similar3D} with ${\mathcal C}_1\equiv {\mathcal C}$ and ${\mathcal C}_2\equiv f({\mathcal C})$, with $f$ again the similarity of \eqref{example2}. 

\begin{table}
\begin{tabular*}{\columnwidth}{@{ }l@{\extracolsep{\stretch{1}}}*{1}{ccccc}@{ }}
\toprule
& \includegraphics[scale=0.08,clip=true,trim = 190 100 190 100]{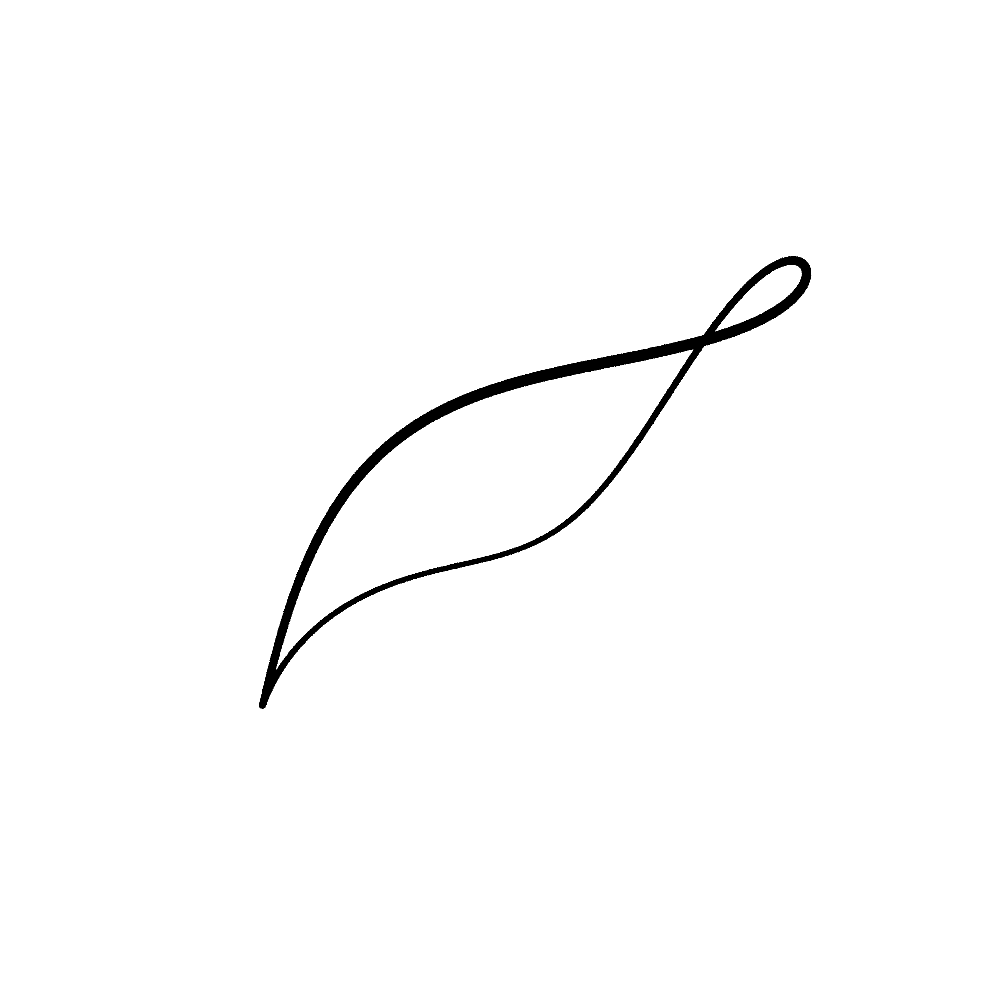} &
  \includegraphics[scale=0.08,clip=true,trim = 190 100 190 100]{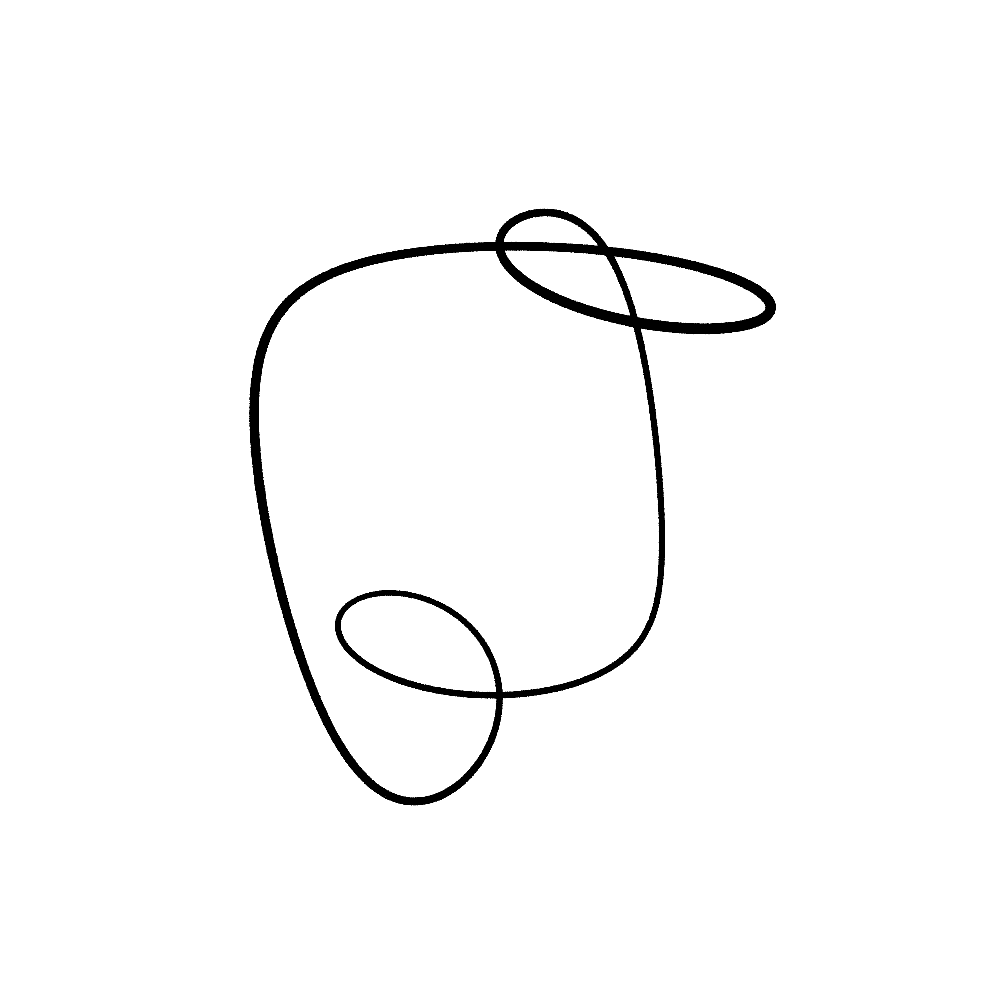} &
  \includegraphics[scale=0.08,clip=true,trim = 190 100 190 100]{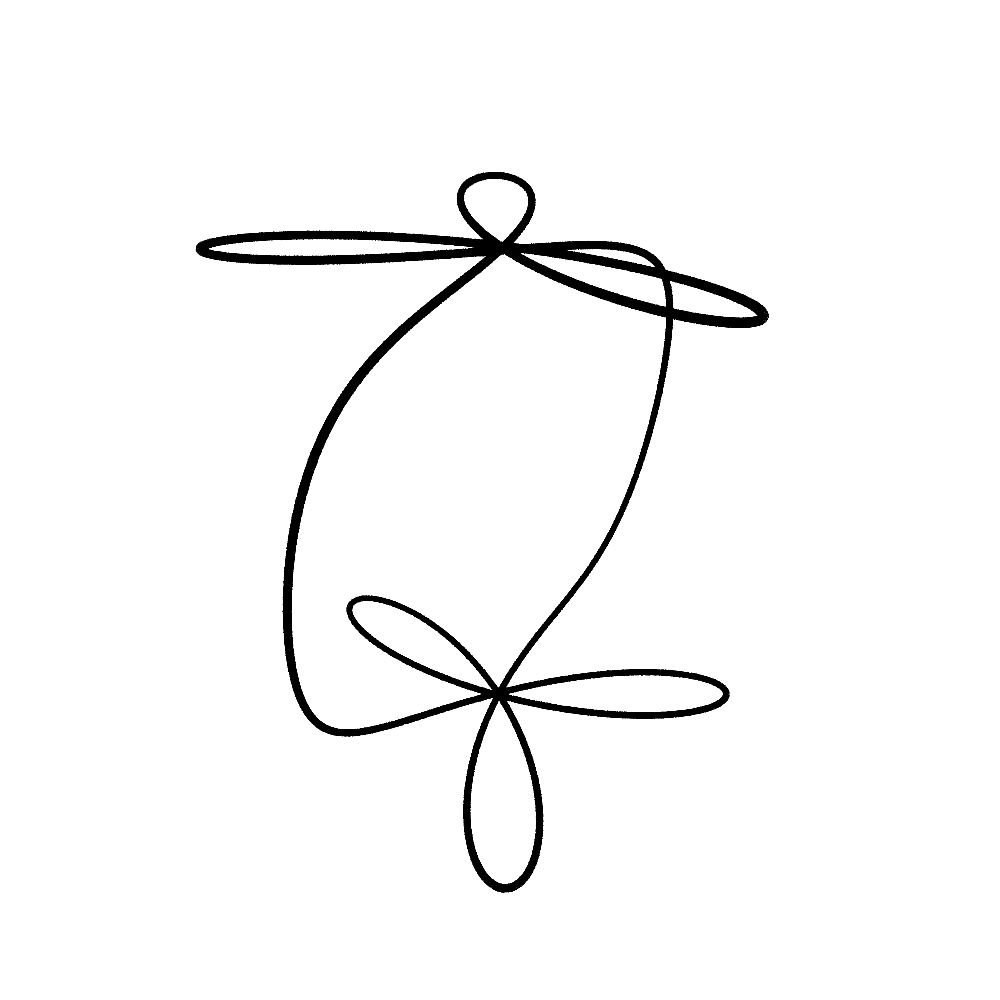} &
  \includegraphics[scale=0.08,clip=true,trim = 190 100 190 100]{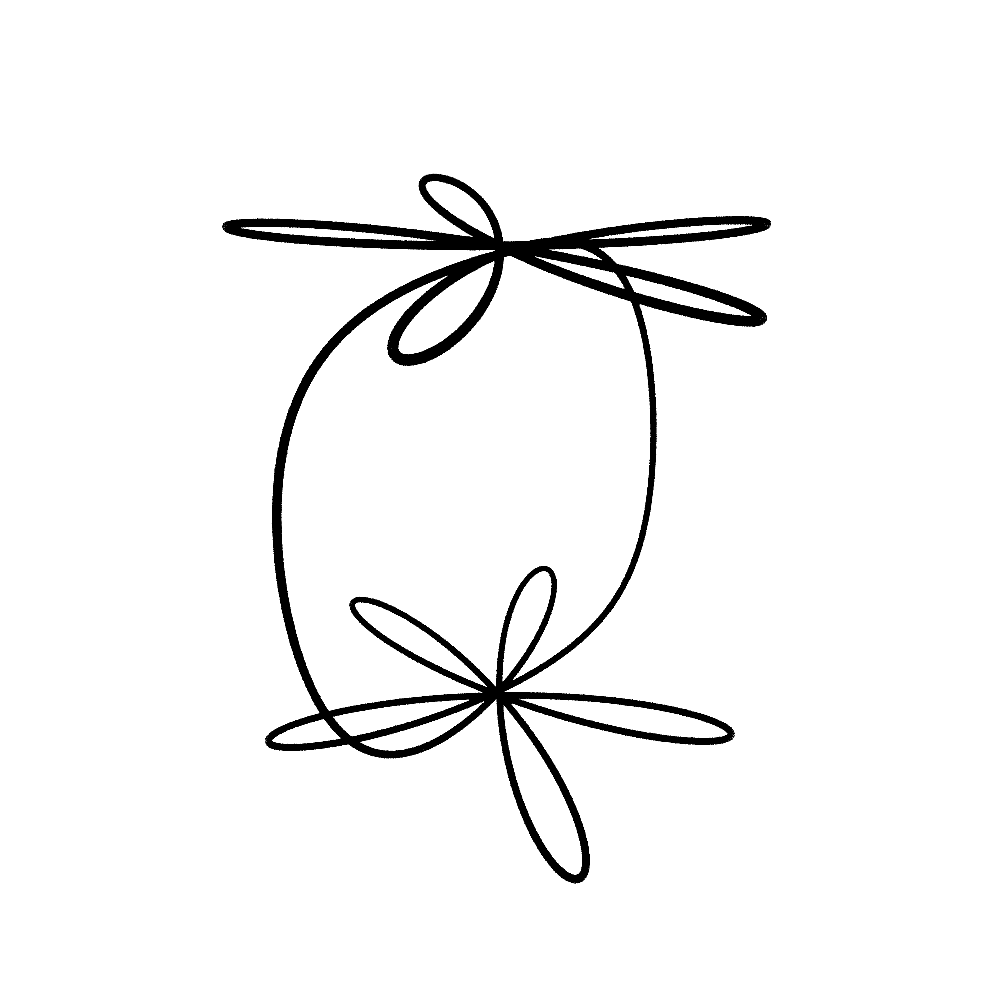} &
  \includegraphics[scale=0.08,clip=true,trim = 190 100 190 100]{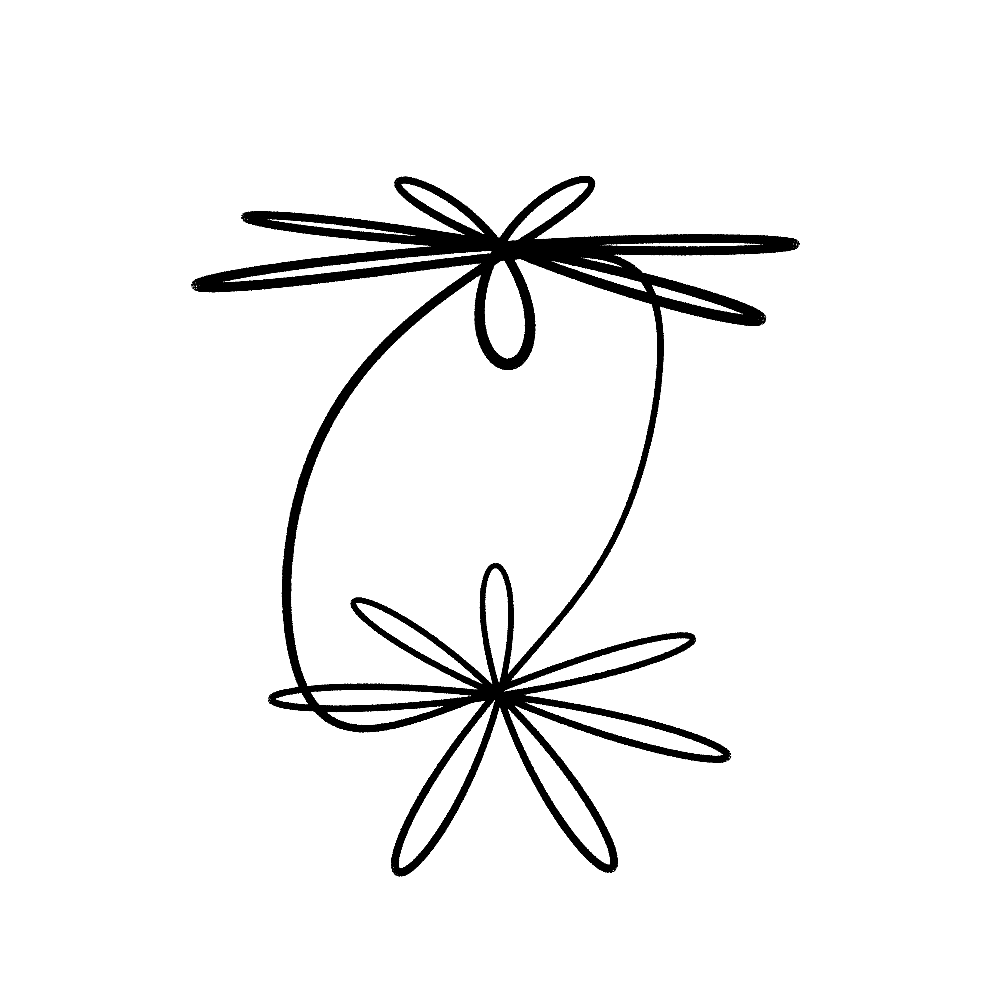} \\
degree   &  $m = 4$ & $m = 8$ & $m = 12$ & $m = 16$ & $m = 20$ \\ 
CPU time &    0.858 &   3.182 &   19.603 &   72.166 &  209.541 \\
\bottomrule
\end{tabular*}
\caption{CPU time (seconds) of Algorithm {\tt Similar3D} applied to daisies of various degrees.}\label{tab:spacerose}
\end{table}

\subsection{Similarities with irrational ratio}
\noindent In the above examples $\lambda$ is rational. In order to test the algorithm in the case where $\lambda \notin \QQ$, we next consider the family of curves 
\begin{equation*}
\bfx_1^n(t) := \left(
\frac{t^{2n+1}}{t^{2n}+1},\frac{t^{2n+3}}{t^{2n}+1},\frac{t^{2n+5}}{t^{2n}+1}
\right),\qquad n = 1, 2, \ldots
\end{equation*}
With the homothety $f({\bf x})=\sqrt{2}\cdot {\bf x}$ and change of parameter $\varphi(t) = t/\sqrt{2}$, for every $n$ the transformed curve
\begin{equation*}
\bfx_2^n(t) := f\circ \bfx_1^n \circ \varphi^{-1} (t) = \left(
\frac{2^{n+1}\cdot t^{2n+1}}{2^n\cdot t^{2n}+1},\frac{2^{n+2}\cdot t^{2n+3}}{2^n\cdot t^{2n}+1},\frac{2^{n+3}\cdot t^{2n+5}}{2^n\cdot t^{2n}+1}
\right)
\end{equation*}
is similar to $\bfx_1^n$ and has rational coefficients as well.

Notice that $\bfx_1^n$ and $\bfx_2^n$ have degree $m = 2n+5$ and coefficients with bitsize ${\mathcal O}(n)$. 
Table \ref{tab:sizedegree2} lists timings for Algorithm {\tt Similar3D} applied to the pairs $(\bfx_1^n, \bfx_2^n)$ for several values of $n$, decomposed as $t = t_\lambda + t_\varphi$ into the timing $t_\lambda$ for the computation of $\lambda$ and $t_\varphi$ for the computation of the tentative M\"obius transformations $\varphi$. Since $\lambda \notin \QQ$, it is necessary to work in an algebraic extension field. Therefore computing the M\"obius transformations $\varphi$, by factoring the gcd $G_{\lambda}$ in \eqref{eq:GR}, requires more time than computing the ratio $\lambda$, which is reflected in Table \ref{tab:sizedegree2}.

Interestingly, the timings $t_\varphi$ are mostly spent not on the M\"obius transformation corresponding to the ratio $\lambda=\sqrt{2}$, but in confirming that other potential, irrational ratios $\lambda$, which are real roots of polynomials of high degree, do not give rise to a similarity. For example, for $n=1$ we need to check a $\lambda$-value which is a root of an irreducible polynomial with rational coefficients of degree 38; for $n=5$, the polynomial has degree 52. Since by Proposition \ref{lambda-unique} the value of $\lambda$ is unique, except perhaps for the sign, after confirming that $\lambda=\sqrt{2}$ gives rise to a similarity between the curves, we can skip the computation for other $\lambda$-values. However, we opted to include all those computations in the table, in order to give an idea of what might happen in other examples with values of the similarity ratio whose minimal polynomial has high degree.

\begin{table}[t!]
\begin{tabular*}{\columnwidth}{@{ }c @{\extracolsep{\stretch{1}}}*{5}{c}@{ }}
\toprule
degree                      & $m=7$ & $m=9$ & $m=11$ & $m=13$ & $m=15$\\ \midrule
$t_\lambda$                 & 0.265 & 0.452 &  0.639 &  1.545 &  2.433\\
$t_\varphi$                 & 1.887 & 1.076 &  2.200 &  2.433 &  7.535\\ \midrule
$t = t_\lambda + t_\varphi$ & 2.152 & 1.528 &  2.839 &  3.978 &  9.968\\
\bottomrule
\end{tabular*}

\caption{CPU time $t$ (seconds) of Algorithm {\tt Similar3D} applied to the pairs $(\bfx_1^n, \bfx_2^n)$ of degree $m = 2n+5$, with $n = 1,2,3,4,5$, decomposed as $t = t_\lambda + t_\varphi$ into the timing $t_\lambda$ for the computation of $\lambda$ and $t_\varphi$ for the computation of the tentative M\"obius transformations $\varphi$.}\label{tab:sizedegree2}
\end{table}

\subsection{Helical curves}
\noindent We also tested Algorithm {\tt Similar3D} for several helical curves. We created these examples (including Example \ref{quint}) by using the results on the generation of cubic and quintic polynomial helices in \cite[\S 23]{Farouki}, as well as the algorithm in \cite{Fatma} for generating general rational helices of any degree. The timings corresponding to these examples are shown in Table~\ref{tab:helical}. These curves are polynomial helices of degree $m\leq 7$.

Table \ref{tab:helical} includes two different pairs of polynomial curves with degree 5: the first one corresponds to the two helical space curves in Example \ref{quint}, which have equal proportionality constants but nevertheless are not similar; the second one corresponds to two similar helical space curves. Although the method in \cite{Fatma} can produce rational, non-polynomial helices, Algorithm {\tt Similar3D} took a long time for even the simplest examples of those. In the case of helical curves, and again with $\lambda \in \QQ$, we observed that the bottleneck of Algorithm {\tt Similar3D} is the use of Gr\"obner bases for eliminating the variable $s$ from the ideal $\langle P_i \rangle_i \subset \RR[s,\lambda]$. 

\begin{table}
\begin{center}
\begin{tabular*}{\columnwidth}{@{ }l@{\extracolsep{\stretch{1}}}*{1}{cccccc}@{ }}
\toprule
degree   & $m = 3$ & $m = 4$ & $m = 5$ & $m = 5$  & $m = 6$ & $m = 7$\\
CPU time &   0.390 &   0.655 &   0.312 &      0.452 &  15.460 &  63.228\\
note     & Ex. \ref{ex:HelicalCurves} & & Ex. \ref{quint}\\
\bottomrule
\end{tabular*}
\caption{CPU time (seconds) of Algorithm {\tt Similar3D} applied to helical polynomial curves of various degrees $m$.}\label{tab:helical}
\end{center}
\end{table}

\section{Similarity of curve segments} \label{segments} 
\noindent In the previous sections we considered \emph{global} rational curves, i.e., curves defined by means of a rational parametrization where the parameter moves over the entire real line (except for some poles). Now let us consider two \emph{curve segments}, i.e., the images of parametrizations of the form
\begin{equation}\label{eq:CurveSegment}
\bfx: I = [a, b] \longrightarrow \RR^3.
\end{equation}
In case these parametrizations are rational, similarity can be detected analogously to the method described in \cite[\S 3.5]{AHM14} for detecting similarity of plane curve segments. We refer the reader to this publication for further detail.

However, in Computer Aided Design the representation \eqref{eq:CurveSegment} is rarely used. Instead, B\'ezier curves, B-spline curves, and NURBS are ubiquitous; see \cite{Patrik, Goldman} for details on the definitions and properties of such curves. These are examples of classes of \emph{curve segments} $\CCC$ with parametrizations taking the form
\begin{equation}\label{eq:CurveSegment2}
\bfx: I \longrightarrow \CCC\subset \RR^3,\qquad \bfx(t) = \sum_{i=0}^\ell \bfc_i B_{i,m}(t),\qquad t\in I = [a, b].
\end{equation}
for certain linearly independent \emph{basis functions} $B_{0,m}, \ldots, B_{\ell,m}$ (with $m$ representing the degree) and \emph{control points} $\bfc_0,\ldots, \bfc_\ell\in \RR^3$. A corresponding \emph{control polygon} $\PPP_\mC$ is represented by the sequence $\mC = (\bfc_0,\ldots, \bfc_\ell)$, where we identify sequences that are reverse to each other, i.e., 
$(\bfc_0,\ldots, \bfc_\ell) \sim (\bfc_\ell,\ldots, \bfc_0)$. It can be visualized as the 1-dimensional piecewise linear subset of $\RR^3$ formed by connecting the consecutive control points by line segments. Notice that for a \emph{fixed} degree $m$ and set $\{B_{0,m}, \ldots, B_{\ell,m}\}$ of linearly independent basis functions, and as a consequence of the linear independence of such functions, the control polygon is unique.

The following properties of \eqref{eq:CurveSegment2} are key to our discussion:
\begin{enumerate}
\item \emph{Properness}. The parametrization $\bfx$ is injective, except perhaps for finitely many parameter values.
\item \emph{Linear independence}. The basis functions $B_{0,m}, \ldots, B_{\ell,m}$ are linearly independent.
\item \emph{Partition of unity}. The basis functions satisfy $B_{0,m}(t) + \cdots + B_{\ell,m}(t) = 1$.
%\item \emph{Uniqueness of the control polygon}. For a given class of proper parametrizations of the form \eqref{eq:CurveSegment2}, there is a one-to-one correspondence between %curve segments and control polygons. %unique control polygon for every curve segment~$\CCC$.
\end{enumerate}

Under the assumption that the basis functions form a partition of unity, the linear nature of the representation \eqref{eq:CurveSegment2} implies
\[
f\circ \bfx (t) = \sum_{i=0}^\ell \mA \bfc_i B_{i,m}(t) + \bfb \sum_{i=0}^\ell B_{i,m}(t) = \sum_{i=0}^\ell f(\bfc_i) B_{i,m}(t)
\] 
for any affine transformation $f(\tx) = \mA \tx + \bfb$. Therefore the following property holds as well.
\begin{enumerate}
\item[4.] \emph{Affine invariance}. For any affine transformation $f$ and parametrized curve segment as in \eqref{eq:CurveSegment2}, one has $f(\PPP_\mC) = \PPP_{f(\mC)}$.
\end{enumerate}

\begin{theorem}\label{thm:similarity-segments}
Let $\bfx_1,\bfx_2$ be two proper parametrizations as in \eqref{eq:CurveSegment2}, where $B_{0,m}, \ldots, B_{\ell,m}$ forms a fixed set of (linearly independent) basis functions, with fixed $m$, forming a partition of unity. Then $\bfx_1,\bfx_2$ are related by a similarity $f$ if and only if the corresponding control polygons $\PPP_{\mC_1},\PPP_{\mC_2}$ are related by $f$.
\end{theorem}

\begin{proof}
``$\Longleftarrow$'': By the hypothesis and affine invariance, $\PPP_{\mC_2} = f(\PPP_{\mC_1}) = \PPP_{f(\mC_1)}$, implying $\bfx_2 = f\circ \bfx_1$ by the uniqueness of the control polygon.

``$\Longrightarrow$'': By the hypothesis, $f\circ \bfx_1$ and $\bfx_2$ have identical image segments. By affine invariance and uniqueness of the control polygon, it therefore follows that $f(\PPP_{\mC_1}) = \PPP_{f(\mC_1)} = \PPP_{\mC_2}$.
\end{proof}

Notice that checking whether or not two polygons in $\RR^3$ are similar is straightforward. Hence, the characterization in Theorem \ref{thm:similarity-segments} can be tested easily. Next we consider the three examples mentioned at the beginning of this section. In each example it is well known that the basis functions are linearly independent and form a partition of unity.

\begin{example}[B\'ezier curves]
If, in \eqref{eq:CurveSegment2}, we take $\ell = m$ and
\[ B_{i,m}(t) := {m\choose i} t^i (1-t)^{m-i},\qquad i = 0,\ldots,m \]
the \emph{Bernstein polynomials} of degree $m$, we obtain the class of \emph{B\'ezier curves} of degree $m$.
Symmetries of B\'ezier curves are studied in \cite{SanchezReyes}.

\end{example}

\begin{example}[B-spline curves]
Suppose $\bft =(t_0,t_1,\ldots, t_{m+\ell+1})$ is a fixed \emph{open knot vector}, i.e., $\ell \ge m$ and $\bft$ takes the form
\[ a = t_0 = \cdots = t_m < t_{m+1} \leq \cdots \leq t_\ell < t_{\ell+1} = \cdots = t_{m+\ell+1} = b,\]
where $t_i < t_{i+m+1}$ for $i = 0, \ldots, \ell$, or more generally an \emph{$(m+1)$-extended} knot vector \cite{Curry.Schoenberg66}. Consider the B-spline basis functions defined recursively by
\[
B_{i,0}(t) = \left\{
\begin{array}{ll}
1, & t_i\leq t< t_{i+1},\\
0, & \text{otherwise},
\end{array}
\right.
\]
and
\[ B_{i,m}(t) = \frac{t - t_i}{t_{i+m} - t_i} B_{i,m-1}(t) + \frac{t_{i+1+m} - t}{t_{i+1+m} - t_i} B_{i+1,m-1}(t), \quad m\geq 1,\ 0\leq i\leq \ell. \]
with the convention that each coefficient is zero when its numerator is zero (also when the denominator is zero). Then \eqref{eq:CurveSegment2} defines a B-spline curve segment on the interval $I = [t_0, t_{m+\ell+1}]$.
\end{example}

\begin{example}[NURBS curves]
Next, let $B_{0,m},\ldots, B_{\ell,m}$ be the B-splines defined above for a given open knot vector $\bft =(t_0,t_1,\ldots, t_{m+\ell+1})$, and let $\omega_0,\ldots, \omega_\ell > 0$ be certain corresponding \emph{weights}. Substituting ``$B_{i,m}$'' in \eqref{eq:CurveSegment2} by the \emph{non-uniform rational B-splines (NURBS)}
\begin{equation}\label{eq:NURBS}
\frac{\omega_j B_{i,m}}{\sum_{j = 0}^\ell \omega_j B_{j,m}},\qquad 0\leq i\leq \ell,
\end{equation}
yields a NURBS curve segment. While such curves are in general rational, the polynomial B-spline segments appear as the special case $\omega_0 = \cdots = \omega_\ell = 1$.
\end{example}

However, Theorem \ref{thm:similarity-segments} cannot be applied when we want to compare two B-spline or NURBS curves with \emph{different} knot vectors or \emph{different} weights. In such a more general setting, we are unaware of a simple characterization of similarity in terms of the control polygons, the knot vectors and the weights; therefore, we pose it here as an open problem.

\section{Conclusion}\label{sec:conclusion}
\noindent We have presented a deterministic algorithm for deciding whether any two rational space curves are related by a similarity, and for determining the similarity in this case. The algorithm exploits the relationship between the curvatures and torsions of two similar space curves and extends the results of \cite{AHM15}, where the problem of detecting the symmetries of rational space curves was addressed. Interestingly, unlike for symmetry detection, it is necessary to distinguish the cases of non-helical and helical curves. In the first case, the experimentation performed so far shows that the algorithm is useful for curves of medium degrees or bitsizes. In the second case, the algorithm is useful for polynomial helices of low degree. 

When the similarity ratio is a rational number, the bottleneck of the algorithm, both for non-helical and helical curves, is the computation of the similarity ratio. This operation depends on the computation of certain resultants in the non-helical case, and on Gr\"obner bases elimination in the helical case, which becomes time-consuming as the degrees or bitsizes grow. When the similarity ratio is irrational, we need to work in an algebraic extension field, and then the computation of the M\"obius transformations may take more time than the computation of the similarity ratio. 

After concluding our discussion on global curves, we briefly considered the similarity of curve segments. For a space of properly parametrized curves satisfying affine invariance and uniqueness of the control polygon, we show that detecting similarity of such curve segments reduces to detecting similarity of the control polygons. We end with three examples of such curve segments, namely B\'ezier curves, B-spline curves of \emph {fixed} degree $m$ for a \emph{fixed} $(m+1)$-extended knot vector, and NURBS for such a knot vector and \emph{fixed} positive weights.

Future work includes seeking alternatives for finding the similarity ratio, as well as detecting symmetries and similarities of implicitly defined algebraic space curves. Additionally, it would be interesting to find alternatives for the special, but important, case of bounded space curves. Finally, the problem of finding a simple condition for the similarity of B-spline curves or of NURBS curves in the more general setting, i.e., with possibly different knot vectors or different weights,is also left as a pending question.

\end{document}